\documentclass[reqno,12pt]{amsart}
\usepackage{amsmath,amssymb,epsfig,color,here,url}

\numberwithin{equation}{section}

\newtheorem{theorem}{Theorem}[section]

\newtheorem{lemma}[theorem]{Lemma}
\newtheorem{proposition}[theorem]{Proposition}
\newtheorem{definition}[theorem]{Definition}
\newtheorem{corollary}[theorem]{Corollary}

\theoremstyle{remark}

\newtheorem{remark}[theorem]{Remark}

\newcommand{\adn}{\widetilde{d}}

\newcommand{\vanish}[1]{}
\newcommand{\Sym}[1]{\operatorname{Sym}\left( #1\right)}

\begin{document}

\title{Counting genus one partitions and permutations}

\author[Robert Cori and G\'abor Hetyei]{Robert Cori \and G\'abor Hetyei}

\address{Labri, Universit\'e Bordeaux 1, 33405 Talence Cedex, France.
\hfill\break
WWW: \tt http://www.labri.fr/perso/cori/.}

\address{Department of Mathematics and Statistics,
  UNC-Charlotte, Charlotte NC 28223-0001.
WWW: \tt http://www.math.uncc.edu/\~{}ghetyei/.}

\date{\today}


\begin{abstract}
We prove the conjecture by M.\ Yip  stating that counting genus one
partitions by the number of their elements and parts yields, up to a
shift of indices, the same array of numbers as counting genus one rooted
hypermonopoles. Our proof involves representing each genus one
permutation by a four-colored noncrossing partition. This representation
may be selected in a unique way for permutations containing no trivial
cycles. The conclusion follows from a general generating function
formula that holds for any class of permutations that is closed under the
removal and reinsertion of trivial cycles. Our method also provides a
new way to count rooted hypermonopoles of genus one, and puts the
spotlight on a class of genus one permutations that is invariant under
an obvious extension of the Kreweras duality map to genus one
permutations.  
\end{abstract}

\maketitle

\section*{Introduction}

Noncrossing partitions, first defined in G.\ Kreweras' seminal
paper~\cite{Kreweras}, have a vast literature in areas ranging from
probability theory through polyhedral geometry to the study of Coxeter
groups. Noncrossing partitions on a given number of
elements are counted by the Catalan numbers, if we also fix the number
of parts, the answer to the resulting counting problem is given by the
Narayana numbers.  
 
A natural generalization of the problem of counting noncrossing
partitions is to count partitions of a given genus. The genus of a
partition may be defined in terms of a topological
representation (see~\cite{CautisJackson} or \cite{Yip} for example), but
there exists also a purely 
combinatorial definition of the genus of a hypermap (thought of as a pair
of permutations, generating a transitive permutation group) that can be
specialized first to hypermonopoles, or permutations (that is, hypermaps
whose first component is the circular permutation $(1,2,\ldots,n)$) and
then to partitions (that, is permutations whose cycles may be written as
lists whose elements are in increasing order). Counting partitions of
a given genus seems surprisingly hard, especially considering the fact
that, for the closely related hypermonopoles, a general machinery was built by
S.\ Cautis and D. \ M. \ Jackson~\cite{CautisJackson} and explicit
formulas were given by A.\ Goupil and G.\ 
Schaeffer~\cite{Goupil-Schaeffer}. It should be noted that 
for genus zero, i.e., noncrossing partitions, the notions of a hypermonopole (in
our language: permutation) and of a partition coincide
(see~\cite[Theorem 1]{Cori}). Thus it seems hard to believe that the two
notions would not only diverge but also give rise to counting problems
of different difficulty in higher genus. Asymptotic estimates for 
the numbers of noncrossing partitions on various surfaces may found
in~\cite{Rueetal}. 

Concerning partitions of a fixed genus, a great deal of numerical evidence
was collected in M.\ Yip's Master's thesis, who made the
following conjecture~\cite[Conjecture 3.15]{Yip}: the number of
genus $1$ partitions on $n$ elements and $k$ parts is the same as the
number of genus one permutations of $n-1$ elements having $k-1$ cycles.
In this paper we prove this conjecture and provide further insight into
the structure of genus $1$ partitions and permutations by representing
them as four-colored noncrossing partitions.

Our paper is structured as follows. After reviewing some basic
terminology and results on the genus of partitions and permutations in
Section~\ref{sec:genus}, in Section~\ref{sec:4cp} we develop a theory of
representing every permutation of genus $1$ by a four-colored
noncrossing partition. The four colors form consecutive arcs in the
circular order and prescribe a relabeling that results in a permutation
of genus at most one. The construction is not unique, but we show that every
permutation of genus $1$ may be represented in such a way. Moreover, as
we show it in Section~\ref{sec:redpp}, if the permutation of genus $1$ is {\em
  reduced} in the sense that it contains no cycle consisting of 
consecutive elements in the circular order (we call these {\em trivial
  cycles}) then we may select a unique four-colored noncrossing
partition representation of our permutation which we call the canonical
representation. This unicity enables us to count reduced permutations
and partitions of genus $1$ in Section~\ref{sec:enured}. We only need to
account for the possibility of having trivial cycles. In
Section~\ref{sec:red} we show how to do 
this, at the level of ordinary generating functions, for any class of
permutations that is closed under the removal and reinsertion of trivial
cycles. Since genus one permutations and partitions form such classes,
we may combine the formula stated in Theorem~\ref{thm:pd} with the
generating function formulas stated in Section~\ref{sec:enured} and
obtain the generating function formulas counting genus $1$ permutations
and partitions with given number of permuted elements and cycles. Since
the resulting formulas stated in Theorems~\ref{thm:pnkgf} and
\ref{thm:pnkgfp} differ only by a factor of $xy$, the validity of M.\
Yip's conjecture is at this point verified. In
Section~\ref{sec:yip} we show how to extract the coefficients from our
generating functions to find the number of partitions of genus $1$. It should
be noted, that our paper thus also provides a new method to count
permutations of genus $1$, whose number was first found by A.\ Goupil
and G.\ Schaeffer~\cite{Goupil-Schaeffer}. The generalized formula stated in
Section~\ref{sec:yip} links the problem of counting genus $1$ permutations
and partitions to the problem of counting type $B$ noncrossing 
partitions, convex polyominos and Jacobi configurations (at least
numerically). The explanation of these connections, 
together with ideas of possible simplifications and further
questions, are collected in the concluding Section~\ref{sec:concl}. 

\section{On the genus of permutations and partitions}
\label{sec:genus}

\subsection{Hypermaps and permutations}
Since the sixties combinatorialists considered permutations as a useful
tool for representing graphs embedded in a topological surface. One of
the main objects in this representation is the notion of a hypermap. 

A hypermap is a pair of permutations $(\sigma, \alpha)$ on a set of
points  $\{1,2, \ldots, n\}$, such that the group they generate is
transitive, meaning that the graph with vertex set $\{1,2, \ldots, n\}$
and edge set $\{i, \alpha(i)\}, \{i, \sigma(i)\} $ is connected.

 It was proved (in \cite{Jacques}) that the number  $g(\sigma,\alpha)$ 
associated to  a hypermap and  defined by:
\begin{equation}
\label{eq:genusdef}
n + 2 -2g(\sigma,\alpha) = z(\sigma) + z(\alpha) + z(\alpha^{-1}
\sigma),
\end{equation}
where $z(\alpha)$ denotes the number of cycles of the permutation $\alpha$,
is a non-negative integer.  It is called the genus of the hypermap.

Taking for $\sigma$ the circular permutation $\zeta_n$ such that for all $i$, 
$\zeta_n(i) = i+1$ (where $n+1$ means $1$) allows to define the genus of
a  permutation  $\alpha\in\Sym{n}$ as that of the hypermap $(\zeta_n, \alpha)$.
Notice that the pair $(\zeta_n, \alpha)$ generates a transitive group
for any $\alpha$ since $z(\zeta_n) = 1$; so that  we may use the
following definition: 

\begin{definition}
\label{def:permgenus}
The genus of a permutation $\alpha$ is the  non-negative integer
$g(\alpha)$ given by:
 $$ n + 1 -2g(\alpha) = z(\alpha) + z(\alpha^{-1} \zeta_n).$$
\end{definition}

 Notice that hypermaps of the form $(\zeta_n, \alpha)$ are often called
hypermonopoles (for instance in \cite{CautisJackson} or \cite{Yip}). 
A different definition of the genus was given in~\cite{Postnikov},
where the genus $h(\alpha)$ of the permutation $\alpha$ is defined as
the genus of the hypermap
$(\zeta_n, \alpha^{-1} \zeta_n \alpha)$. In this definition a permutation 
is of genus $0$ if and only if it is a power of $\zeta_n$; in  ours permutations
of genus $0$ correspond to noncrossing partitions, a central
object in combinatorics.

\subsection{Partitions of the set $\{1,2, \ldots, n\}$}
To  a partition $P=(P_i)_{i=1,k}$ of the set $\{1,2,\ldots n\}$ is associated
the permutation $\alpha_P$ which has $k$ cycles, each one  corresponding
to one of the $P_i$ written with the elements in increasing order. This allows 
to define the genus of the partition $P$ as that of the permutation
$\alpha_P$. 

It was shown in  \cite[Theorem 1]{Cori} that a permutation
$\alpha$ is  
 of genus $0$, if and only if there exists a noncrossing partition
$P$ such that $\alpha = \alpha_P$.

A noncrossing partition may be drawn as a circle on which we put 
the  points $1,2, \ldots , n$ in clockwise order and parts of size $p> 2$
are represented with $p$-gons inscribed in the circle, parts of size $2$ 
by segments, and parts of size $1$ by isolated points.

The partition $P= (\{1,5, 7 ,8\} , \{2,4\}, \{3\}, \{6\})$ 
is represented in Figure~\ref{fig:ncpart} below.

\begin{figure}[h]
\begin{center}
\includegraphics{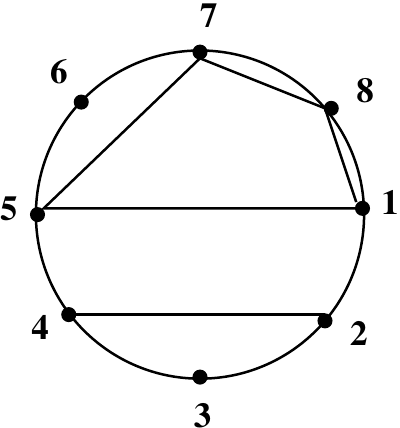}
\end{center}
\caption{The noncrossing partition $P$}
\label{fig:ncpart}
\end{figure}

\subsection{The genus and the cycle structure}
Since the genus of a permutation $\alpha$ is a function of $z(\alpha)$,
the number of its cycles, in the sequel we will consider permutations as
products of their cycles, study their structure, and the effect of minor
changes on the cycle structure. In particular, we will be interested in
the change of the genus when we compose a permutation with a single {\em
  transposition}. A transposition $\tau\in \Sym{n}$, 
exchanging the two points $i, j$,  will be denoted by $\tau = (i, j)$. 
It has $n-2$ cycles of length $1$ and one of length $2$, hence $z(\tau) = n-1$.
Note that we compose permutations right to left, i.e., 
we define the product $\alpha \beta$ of two permutations as the
permutation which sends $i$ into $\alpha(\beta(i))$.

We will often use the following Lemma:

\begin{lemma}
\label{lemma:transpositions}
The number of cycles of the products $\tau\alpha$ and $\alpha\tau$ 
of a permutation $\alpha$ and a transposition
$\tau = (i, j)$ differs from the number of cycles of $\alpha$ by $1$. 
The sign of the change depends on whether $i$ and $j$ belong to the
same cycle of $\alpha$ or not. We have 
$$
z(\tau\alpha)=z(\alpha\tau)=
\left\{
\begin{array}{ll}
z(\alpha)+1 & \mbox{if $i$ and $j$ belong to the same cycle of
  $\alpha$;}\\
z(\alpha)-1 & \mbox{if $i$ and $j$ belong to different cycles of $\alpha$.}\\
\end{array}
\right.
$$ 
\end{lemma}

\begin{definition}
Two cycles in a permutation $\alpha$ are {\em crossing} if there exists
two elements $a, a'$ in one of them  and $b,b'$ in the other such that
$a< b < a' < b'$. 
\end{definition}

Observe  that if such elements exist they may be taken 
such that $a' = \alpha(a)$ and   $b' = \alpha(b)$.

An element $i$ of ${1,2,\ldots ,n}$ is a {\em  back point} of the permutation
$\alpha$ if $\alpha(i) < i$ and $\alpha(i)$ is not the smallest element in
its cycle (i. e. there exist $k>1$ such that $\alpha^k(i) <\alpha (i)$).

\begin{definition}

 A {\em twisted cycle} in a permutation $\alpha$ is a cycle $(b_1, b_2,
 \ldots,  b_p)$ containing a back point. 
 \end{definition}

  The genus of a permutation may  be determined by counting  back points
  as the following variant of \cite[Lemma 5]{Chapuy} shows. 
 
 \begin{lemma}
\label{lemma:nbBackPoints}
For any  permutation $\alpha\in \Sym{n}$, 
the sum of the number of back points of the permutation $\alpha$ and the
number of those of $\alpha^{-1} \zeta_n$ is equal
to $2g(\alpha)$. 
\end{lemma}

\begin{proof}
As usual, for a permutation $\alpha\in \Sym{n}$, let $\mbox{EXC}(\alpha)$
denote the set of {\em excedances} of $\alpha$, i.e., the set of
elements $i$ such that $\alpha(i)>i$. The number of back points of
$\alpha$ is then $n-|\mbox{EXC}(\alpha)|-z(\alpha)$. After replacing
$2g(\alpha)$ with its expression in Definition~\ref{def:permgenus},
our lemma is equivalent to 
$$
|\mbox{EXC}(\alpha)|+|\mbox{EXC}(\alpha^{-1} \zeta_n)|=n-1.
$$  
To prove this equation observe first that, for all $i$ satisfying $i\neq
\alpha^{-1}(1)$, the relation  
$i\in \mbox{EXC}(\alpha)$ is equivalent to $\alpha(i)-1\not
\in \mbox{EXC}(\alpha^{-1} \zeta_n)$. 
Thus the number of excedances
of $\alpha$ in the set $\{\alpha^{-1}(2),\ldots,\alpha^{-1}(n)\}$   
plus the number of excedances
of $\alpha^{-1} \zeta_n$ in the set $\{1,\ldots,n-1\}$ is $n-1$. Finally
$\alpha^{-1}(1)$ is not an excedance of $\alpha$ and $n$ is not an
excedance of any permutation in $\Sym{n}$.
\end{proof}

 Notice that a permutation
is  associated to a partition if and only if it contains no twisted cycle,
moreover the partition and the associated permutation are of genus $0$
if and only if  
there are no crossing cycles. Noncrossing partitions were extensively
studied (see for instance \cite{Simion-noncrossing}).

\section{Genus one  permutations and four-colored noncrossing 
partitions}
\label{sec:4cp}

We define a four-coloring of a noncrossing partition of the set
$\{1,2,\ldots,n\}$ as a partitioning 
of the $n$ points on the circle into four arcs denoted $A$, $B$, $C$,
$D$ in clockwise order where $A$ is the arc containing the point $1$ and in
which $C$ is only arc allowed to contain no point. We will denote by
$\gamma = (A, B, C, D)$ such 
a $4$-coloring. Equivalently a four-coloring may be defined by $4$
integers defining the numberings  of the points  in the four arcs. These are 
$1 \leq i < j \leq k < \ell \leq n$, giving
\begin{equation}
\label{eq:ABCD}
\begin{array}{ll}
A = \{\ell+1,  \ldots, n, 1, \ldots, i\}, &
B = \{i+1,  \ldots,   j\},\\  
C = \{j+1,  \ldots,   k\}, &
D = \{k+1,  \ldots,   \ell\},\\
\end{array}
\end{equation}  
where
$C$ is empty when $j=k$. In this notation, $A = \{1, \ldots, i\}$ holds
when $\ell = n$.

\begin{definition}
\label{def:coloringp}
We call the sequence $(i,j,k,\ell)$, marking the right endpoints of the
color sets in (\ref{eq:ABCD}), a {\em sequence of coloring points} of
the partition $P$. 
\end{definition}

\medskip

To any four-colored noncrossing partition $(P, \gamma)$ (where $\gamma=
(A, B, C, D)$ ) we associate a permutation $\alpha = \Phi(P, \gamma)$ in
which cycles are obtained from the parts of $P$ by renumbering
the points in the following way: 

\smallskip

We leave the numbering of the points in $A$  unchanged and we continue
labeling in such a way that the elements of $A$ are followed by the
points in $D$, then by the points in $C$, and finally by the points in
$B$. Within each color set, points are numbered in clockwise order.
Thus the elements of $A$ are numbered with $\ell+1, \ell+2, \ldots , n, 1, 2,
\ldots i$, the elements of $D$ are numbered from $i+1$ to $i+\ell-k$,
the elements of $C$ are numbered from $i+\ell-k +1$ to $i+\ell-j$ an the
elements of $B$ are numbered from $i+\ell - j+1$ to $\ell$.  
After introducing 
\begin{equation}
\label{eq:abcd}
a = i,\, b = i + \ell -k,\, c = i+\ell - j,\quad\mbox{and}\quad d=\ell, 
\end{equation}
we obtain that the color sets, in terms of the relabeled elements, are
given by 
\begin{equation}
\label{eq:newABCD} 
\begin{array}{rcl}
A & = &
\left\{
\begin{array}{ll}
\{1, 2, \ldots a, d+1, \ldots, n\} & \mbox{if $d \neq n$,}\\
\{1, 2, \ldots, a\}& \mbox{otherwise;}\\
\end{array}
\right.
\\
B &=& \{c+1, c+2, \ldots , d\}; \quad 
D = \{a+1, a+2, \ldots , b\};\\
C &=& 
\left\{
\begin{array}{ll}
\{b+1, b+2, \ldots , c\} & \mbox{if $c \neq b$,}\\
\emptyset & \mbox{otherwise.}\\
\end{array}
\right.
\end{array}
\end{equation}
Let us also note for future reference that the linear map taking
$(i,j,k,\ell)$ into $(a,b,c,d)$ is its own inverse, i.e., we have 
\begin{equation}
\label{eq:ijkl}
i=a,\, j=a+d-c,\, k=a+d-b \quad \mbox{and}\quad \ell=d.
\end{equation}

Once the points are renumbered, each cycle of $\alpha $ is  obtained
from a  part $P_q = \{x_1, x_2, \ldots x_p\}$    of $P$ by writing the
numbering  of the corresponding points $x_1, x_2, \ldots x_p$, where the
$x_i$'s are in clockwise order. 

\begin{figure}[h]
\label{fig:renumber}
\begin{center}
\includegraphics{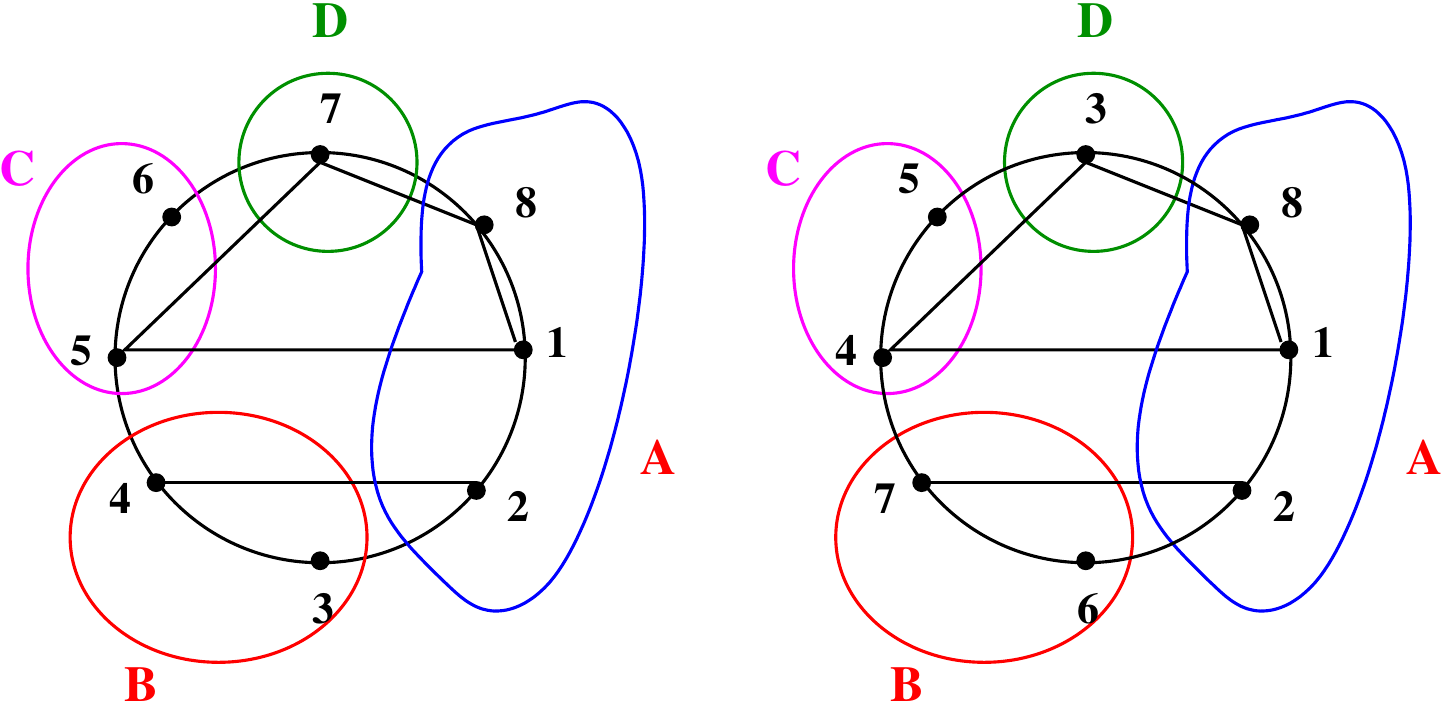}
\end{center}
\caption{A four-coloring of $P$ and the induced renumbering of points}
\label{fig:renumbering}
\end{figure}

 For the example shown in Figure~\ref{fig:renumbering} we obtain the
following permutation of genus $1$: 
$$ \alpha \ = \ \Phi(P, \gamma)  = (1,4,3,8) (2,7) (5) (6)$$

In the sequel  it will be convenient to say that a point $p$  has color
$X$ for $X = A, B, C, D$ if $p \in X$, a part $P_q$ will be unicolored,
bicolored, three-colored or four-colored depending on the number of
different colors its points have. 
\begin{remark}
\label{rem:colors}
A unicolored part of a noncrossing partition $P$ gives rise to  a cycle
in $\Phi(P,\gamma)$ which does not cross any other cycle and is not twisted.
A bicolored  part with points in two different colors $X$ and $Y$ is not
twisted but it crosses any cycle coming from a part that has points of
color $X$ as well as at least one point whose color is neither $X$ nor
$Y$. A bicolored part with points of color $X$ and $Y$ does not cross a
any part that is contained in or disjoint from
$X\cup Y$. A three or four-colored part gives rise to a
twisted cycle. 
\end{remark}

The main point in this section is the following characterization:

\begin{theorem}
\label{thm:4cg1}
If $(P, \gamma)$ is a four-colored noncrossing partition then 
$\Phi(P, \gamma)$ is a permutation of genus $0$ or $1$. It is of genus $1$ if
and only if at least one of these two conditions is satisfied: 

\begin{enumerate}
\item There exists a part $P_q$ which is three or four-colored.

\item There exists two  parts $P_q, P_r$ which are two colored and share
  a common color, more precisely there are 
three different colors $X, Y, Z$ such that
$$ P_q \cap X   \neq \emptyset, \ \   P_q \cap Y \neq \emptyset,\ \ 
P_q\subseteq X\cup Y\quad {\rm and} \quad  P_r \cap X   \neq \emptyset,
\ \ P_r\cap Z \neq \emptyset, \ \  P_r\subseteq X\cup Z.$$
\end{enumerate} 
\end{theorem}

\begin{proof}
Let $i, j, k, \ell$ define the four-coloring $\gamma$ and let 
$\beta$ be the permutation associated to the partition $P$, set
$\alpha = \Phi(P, \gamma)$. The renumbering of the points around the circle
may be considered in two ways:

\medskip
The first way is conjugation. Consider the permutation $\phi$ that
  takes each $i$ into its new label after the renumbering operation. We
  then have $\alpha = \phi \beta \phi^{-1}$. Note that $\phi$ is given by
the coloring points $(i,j,k,\ell)$ via the formula 
\begin{equation}
\label{eq:phi}
\phi(x)=
\left\{
\begin{array}{rl}
x & \mbox{if $x \in A$};\\
x+\ell-j & \mbox{if $x \in B$};\\
x+i+\ell -j- k & \mbox{if $x \in C$};\\
x+i -k & \mbox{if $x \in D$.}\\
\end{array}
\right.
\end{equation}
Although this formula is unimportant for this proof, we will have good
use of it later in the proof of the converse of our present statement. Now let
$\theta = \phi \zeta_n \phi^{-1}$, since conjugation does not change the
number of cycles we have: 
 
\begin{equation}
\label{eq:genusTheta}
g(\zeta_n, \beta) = g(\theta, \alpha) = 0.
\end{equation}
Since $\theta$ has only one cycle, just like $\zeta_n$, the above
equation, together with formula~(\ref{eq:genusdef}) yields
\begin{equation}
\label{eq:ztheta}
n+1-z(\alpha)=z(\alpha^{-1}\theta).
\end{equation} 

\medskip
The second way is multiplication by transpositions. It is easy 
to check that 
\begin{equation}
\label{eq:theta}
\theta = (1,2, \ldots, a, c+1, \ldots, d, b+1, \ldots, c, a+1, \ldots,
b, d+1, \ldots, n),
\end{equation}
where $(a,b,c,d)$ is given by (\ref{eq:abcd}),  hence
$\theta=\zeta_n (a,c)(b,d)$.   
We are now able  to compute the genus of the permutation $\alpha$.
By Definition~\ref{def:permgenus} we have
$$
2g(\alpha)=n+1-z(\alpha)+z(\alpha^{-1}\zeta_n).
$$ 
Using (\ref{eq:ztheta}) we may rewrite the last equation as 
$$ 2g(\alpha) = z(\alpha^{-1} \theta)- z(\alpha^{-1} \zeta_n).$$
But since $\alpha^{-1} \theta$ is obtained from $\alpha^{-1} \zeta_n$
by multiplying by two transpositions, by Lemma \ref{lemma:transpositions}, 
the difference of their number of cycles is $0$, $2$ or $-2$. Since the
genus is a non-negative integer we have that $g(\alpha)$ is $0$ or $1$.
If any of the  conditions given above are satisfied  then  $\alpha$ has a
twisted cycle or two crossing cycles hence it cannot be of genus $0$,
ending the proof, if none of them is satisfied then  
$\alpha$ has no twisted cycle and no two crossing cycles, it is then of
genus $0$ (a permutation of a noncrossing partition).
\end{proof}

\noindent
To state a converse of Theorem~\ref{thm:4cg1} we introduce the following
notion:
\begin{definition}
\label{def:sepp}
Let $\alpha$ be a permutation of genus $1$. We say that
the sequence of integers $(a, b, c, d)$ is a {\em sequence
  of separating  points} for $\alpha$ if the permutation $\theta =\zeta_n
(a,c)(b,d)$ is such that the  genus of the hypermap
$(\theta,\alpha)$ is zero and 
\begin{equation}
\label{ineq:abcd}
a  < b \leq c < d.
\end{equation}
\end{definition}

Notice that (\ref{ineq:abcd}) implies that $\theta$ is a circular permutation.
Equations~(\ref{eq:genusTheta}) and (\ref{eq:theta}) have the following
consequence. 
\begin{remark} 
\label{rem:colsep}
If a permutation $\alpha$ of genus $1$ is represented as $\alpha=\Phi(P,
\gamma)$ by a four-colored noncrossing partition $(P,\gamma)$ then the
sequence of coloring points $(i,j,k,\ell)$ gives rise to the sequence of
separating points $(a,b,c,d)$ given by (\ref{eq:abcd}).
\end{remark}

\begin{proposition}
\label{prop:coloringp}
Let $\alpha$ be a permutation of genus $1$ on $n$ elements that has a sequence of
separating points $(a, b, c, d)$. Then there is a noncrossing partition 
$P$ and a four-coloring $\gamma = (A,B,C,D)$ representing $\alpha$ as 
$\alpha = \Phi(P, \gamma)$ whose sequence of coloring points
$(i,j,k,\ell)$ is obtained from $(a,b,c,d)$ via (\ref{eq:ijkl}).
\end{proposition}
\begin{proof}

Since $\theta =\zeta_n (a,c)(b,d)$ is circular, there is a permutation
$\phi$ satisfying $\phi \zeta_n \phi^{-1} =\theta$. We make this map
$\phi$ unique by requiring $\phi(1)=1$. It is easy to verify that $\phi$
is given by (\ref{eq:phi}) for the sequence $(i,j,k,l)$ given by
(\ref{eq:ijkl}). The permutation $\beta =
\phi^{-1} \alpha \phi$ satisfies 
$$g(\zeta_n, \beta) = g(\phi^{-1} \theta \phi, \phi^{-1} \alpha \phi)=
g(\theta,\alpha)= 0,$$ 
hence $\beta$ determines a noncrossing partition $P$. As a consequence
of (\ref{eq:ABCD}) and (\ref{eq:phi}), the four-coloring $\gamma$
associated to $(i,j,k,\ell)$ satisfies $\alpha = \Phi(P, \gamma)$.
\end{proof}

\begin{definition}
We call the representation described in Proposition~\ref{prop:coloringp}
{\em the four-colored noncrossing partition
representation induced by the sequence of separating points $(a,b,c,d)$}.
\end{definition}

Now we are ready to state the converse of Theorem~\ref{thm:4cg1}.
\begin{theorem}
\label{theorem:existColor}
For any permutation $\alpha$ of genus $1$, there exists a noncrossing 
partition $P$ 
and a four-coloring $\gamma$ such that $\alpha = \Phi(P, \gamma)$.
\end{theorem} 

\begin{proof}
By Proposition~\ref{prop:coloringp} it suffices to show that every
permutation of genus $1$ has a sequence $(a,b,c,d)$ of separating points.
Let $\alpha$ be a permutation of genus $1$, then the permutation
$\alpha' = \alpha^{-1}\zeta_n$ is also of genus $1$, thus $\alpha'$ has two
crossing cycles or a twisted cycle or both. 

\begin{enumerate}
\item
If $\alpha'$ has two crossing cycles then one of these cycle contains two points $a,c$
and the other one two points $b,d$ such that $a< b < c < d$.

By (\ref{ineq:abcd}), $\theta = \zeta_n(a,c)(b,d)$ is circular. Moreover
$\alpha^{-1}\theta$ is obtained from $\alpha^{-1}\zeta_n$ by multiplying
it by two transpositions exchanging elements belonging to the same cycle,
hence $ z(\alpha^{-1}\theta) = z(\alpha^{-1}\zeta_n) +2$.
By the definition of the genus, since $z(\theta ) = z (\zeta_n)$, we get
$g(\theta, \alpha) = g(\alpha) - 1 = 0$.

\item If $\alpha'$ has a twisted cycle, this can be written
$(a, x_1, \cdots , x_p, d, b, y_1, \cdots y_q),$  where
$a$ is the smallest element of the cycle and $d>b$, 
giving $ a< b < d$.
Consider the two transpositions $(a,b)$ and $(b,d)$
It easy to check that the product 
$\theta =\zeta_n(a,b)(b,d)$ is equal to:  $ (1, 2, \cdots a, b+1, \cdots d, a+1, \cdots b, d+1, \cdots n)$.
Moreover, the permutation $\alpha'(a,b)(b,d)$ has the same cycles
as $\alpha'$ except the one containing $a,b,d$ which is broken into three
cycles:

$$ (a, y_1, \cdots y_q)  \   (b) \  (d, x_1, \cdots , x_p),$$

showing that again $$ z(\alpha^{-1}\theta) = z(\alpha^{-1}\zeta_n) +2$$
and $g(\theta, \alpha) = g(\zeta_n, \alpha) - 1 = 0$ hold.

\end{enumerate}
We obtained that, in the first case $(a,b,c,d)$, and in the second case
$(a,b,b,d)$, is a sequence of separating points for $\alpha$.
\end{proof}

It is easy to detect in a four-colored noncrossing partition representation
of a permutation of genus $1$ whether it is a partition, or whether it
has twisted cycles, as we will see in the following observations.

\begin{corollary}
\label{c:genus1part}
A permutation $\alpha$ of genus $1$ is a partition if and only if it  
may be represented by a four-colored noncrossing partition $(Q,
\gamma)$ that has no three or four-colored part and has at least two
two-colored parts. 
\end{corollary}
Indeed, a three or four-colored part would give rise to a twisted cycle 
which partition can not have. Without twisted cycles, a permutation of
genus $1$ must have a pair of crossing cycles which can only be
represented by two-colored parts. 
To state our next observation, we introduce the notion of
simply and doubly twisted cycles.

\begin{remark}
\label{rem:3color}
For future reference we also note that every genus $1$ partition
$\alpha\in\Sym{n}$ has a three-colored 
non-crossing partition representation, that is, a four-colored
representation with $C=\emptyset$. Indeed, since $\alpha$ does not have
any back point, by Lemma~\ref{lemma:nbBackPoints}, $\alpha^{-1}\zeta_n$
must have two back points. We may use the construction presented in the
second case of the proof of Theorem~\ref{theorem:existColor} to
construct a three-colored noncrossing partition. A variant of this
observation was also made in \cite[p.\ 63]{Yip}.
\end{remark}

\begin{definition}
A cycle of $\alpha$ is simply twisted if
contains exactly one back point and 
it is doubly twisted if it has two back points.
\end{definition}

\begin{remark}
In a four-colored noncrossing partition representation of a permutation
of genus $1$, three colored parts correspond to simply twisted cycles
and four-colored parts correspond to doubly twisted
cycles.
\end{remark}

\begin{proposition}
\label{prop:4types}
In a permutation $\alpha$  of genus $1$, all cycles are either not twisted
or simply or doubly twisted. Moreover, exactly one of the
following assertions is satisfied:
\begin{enumerate}
\item   $\alpha$  has no twisted cycle, hence it corresponds to a partition;
\item $\alpha$   has a unique simply twisted cycle;
\item $\alpha$   has a unique doubly twisted cycle;
\item $\alpha$  has two simply twisted cycles.
\end{enumerate}
There is an example of a permutation of genus $1$ of each of the above
four types. 
\end{proposition}
\begin{proof}
\begin{figure}[h]
\includegraphics{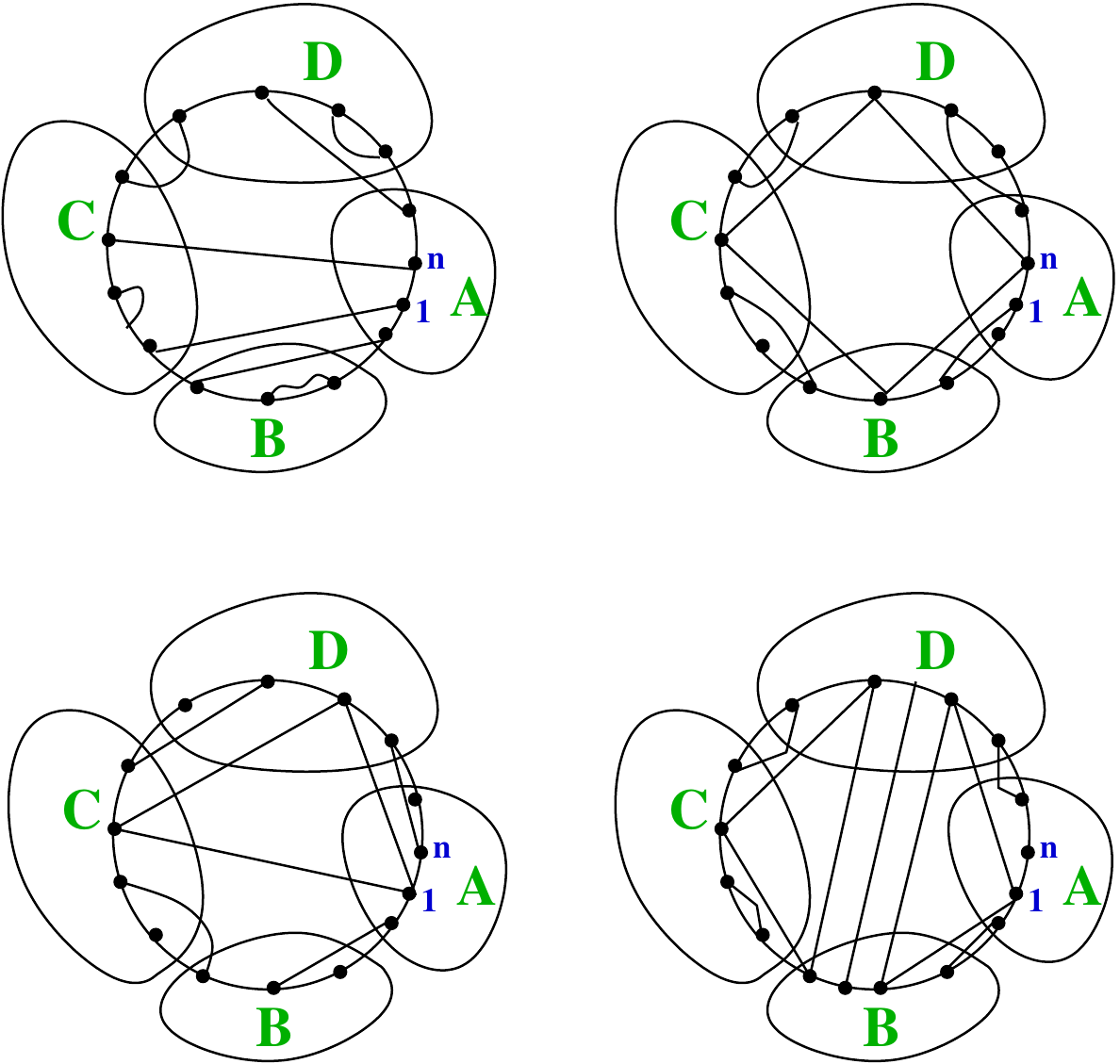}
\caption{The four types of genus $1$ permutations}
\label{fig:4cases}
\end{figure}
By Lemma~\ref{lemma:nbBackPoints}, a permutation of genus $1$ may have at
most two back points. If $\alpha$ has no back points then it is a
partition. If it has one back point then it has a unique simply twisted
cycle. If it has two back points, then these are either on the same
(doubly twisted) cycle, or on two separate (simply twisted cycle). 
An example of a permutation of each type is sketched using a
four-colored noncrossing partition representation in 
Figure~\ref{fig:4cases}. 
\end{proof}

\section{Reduced permutations and partitions}
\label{sec:redpp}

\begin{definition}
\label{def:trivcycle}
A {\em trivial} cycle  in a permutation is a cycle  consisting of consecutive 
points on the circle,
{\em i. e.}  a cycle $C_i = (i, i+1,  \ldots,  i+p)$ where sums
 are taken modulo $n$. A permutation is
{\em reduced} if it contains no trivial cycle.
\end{definition}

\begin{lemma}
\label{lemma:planarPartition}
Let $\theta$ and $\alpha$ be two permutations in $\Sym{n}$ 
 such that $\theta$ is circular and
$g(\theta, \alpha) = 0$. If an integer   $x$ satisfies
$$\alpha(x) = \theta^k(x) \ \ \ \ \ \ \ \ {\rm for } \ \ \ 1 < k <  n$$
then there exists a cycle of $\alpha$ consisting of consecutive
points in the sequence
$$ \theta(x),  \theta^2(x), \ldots, \theta^{k-1}(x) $$
\end{lemma}
\begin{proof}
Use conjugation by a permutation $\phi$ such that 
$\phi\theta\phi^{-1} = \zeta_n$. Then the statement follows by repeated
use of the following, trivial observation: if a noncrossing partition
contains a part $a_1 < a_2 < \dots < a_p$ such that one of the $a_i$'s 
satisfies  $a_{i+1} > a_i +1$ then there is another part contained in
the set $\{a_i+1, a_i+2,\ldots, a_{i+1}-1\}$. Applying the same
observation repeatedly, we end up with a part consisting of consecutive
integers greater than $a_i$ and less  than $a_{i+1}$. 
\end{proof}

As a consequence of Lemma~\ref{lemma:planarPartition}, a permutation $\alpha$
of  genus $1$ is reduced if and only if each of its cycles either
crosses another one or it is twisted. Indeed, by Remark~\ref{rem:colors}, a
cycle that does not cross any other cycle and is not twisted corresponds
to a unicolored part in a four-colored noncrossing partition
representing $\alpha$ and, by Lemma~\ref{lemma:planarPartition}, the
same color set contains a part consisting of consecutive points, which
represents a trivial cycle. Thus the representation of a reduced
$\alpha$ can not have unicolored parts.  

We now define for a reduced permutation $\alpha$ of genus $1$
a canonical sequence  of separating points and 
the canonical representation of it as a four-colored noncrossing partition.

\begin{definition}
\label{def:cancs}
Let $\alpha$  be a reduced  permutation of genus $1$. The {\em canonical
  sequence of separating points $(a,b,c,d)$} of $\alpha$ is defined as
follows: 
\begin{enumerate}
\item  $a$ is the smallest integer such that  $\alpha(a) \neq a+1$; 
\item $b$ is the smallest integer  satisfying $b>a$ and such that either
  $\alpha(b) > \alpha(a)$ or $\alpha(b) \leq a$ holds;
\item $c = \alpha(a) -1$;
\item $d = n$ if $\alpha(b) = 1$ and $d = \alpha(b)-1$ otherwise.
\end{enumerate}
We call the four-colored noncrossing partition representation induced by
the canonical sequence of separating points the {\em canonical
  representation} of $\alpha$.
\end{definition}

In the proof of Proposition~\ref{prop:cancs} below we will show that
the canonical sequence of separating points exists, it is unique, and it
is indeed a sequence of separating points, giving rise to a four-colored
noncrossing partition representation. Our proof relies on the following lemma.
\begin{lemma}
\label{lem:split}
Let $\alpha$ be a permutation of $\Sym{n}$  such that for some $a$ satisfying $a +1 < \alpha(a)$,
the set $X_1= \{a+1, a+2, \ldots, \alpha(a) -1\}$ is a union of cycles of $\alpha$. Then $\alpha$ may be split
into two permutations $\alpha_1$ acting on $X_1$ and $\alpha_2$ acting on 
$X_2 = \{1, 2, \ldots, n\} \setminus X_1$ such that
$$ g(\alpha) =   g(\alpha_1) +  g(\alpha_2)$$
\end{lemma}

\begin{proof}
Let $n_1$ be the number of elements of $X_1$ and $n_2$ be that of $X_2$. Consider the transposition 
$\tau$ exchanging $a$ and $c=\alpha(a) -1$, then $\zeta_n \tau$ has two
cycles of lengths $n_1$ and $n_2$ respectively, permuting the elements
of $X_1$ and $X_2$ respectively.  Since $\alpha^{-1}\zeta_n(c) =a$, we have:
$$ z(\alpha^{-1}\zeta_n\tau) =  z(\alpha^{-1}\zeta_n) +1$$
Moreover $$ z(\alpha) = z(\alpha_1)  + z(\alpha_2) \hskip 0.2cm  \makebox{\rm and}   \hskip 0.2cm z(\alpha^{-1}\zeta_n\tau) = 
z(\alpha_1^{-1}\zeta_{n_1}) + z(\alpha_2^{-1}\zeta_{n_2})$$
where $\zeta_{n_1}=(a+1,a+2,\ldots,\alpha(a)-1)$ and $\zeta_{n_2}$ is
the analogous circular permutation on $X_2$. Computing the genus of
$\alpha_1$ and $\alpha_2$ we get: 
$$ 2g(\alpha_1) = n_1 - z(\alpha_1) - z(\alpha_1^{-1}\zeta_{n_1}) -1  \hskip 0.2cm  \makebox{\rm and}  \hskip 0.2cm 
2g(\alpha_2) = n_2- z(\alpha_2) - z(\alpha_2^{-1}\zeta_{n_2}) -1$$
Adding the two equations and using the preceding relations we get:
$$2(g(\alpha_1) ) + g(\alpha_2) ) = n_1 + n_2 - z(\alpha) - z(\alpha^{-1}\zeta_n\tau) -2$$
Since $n_1 + n_2 = n$ and $z(\alpha^{-1}\zeta_n\tau) = z(\alpha^{-1}\zeta_n) +1$ we obtain the expected 
relation between the genuses of $\alpha, \alpha_1, \alpha_2$. 
\end{proof}

\begin{proposition}
\label{prop:cancs} 
Every reduced permutation of genus $1$ of $n$ elements has a unique
canonical sequence $(a,b,c,d)$ of separating points, that induces a
four-colored noncrossing partition representation. 
\end{proposition}
\begin{proof}

It is easy to see that an  element $a$ as defined above exists since 
if $\alpha(i) = i+1$ for all $i < n$ then $\alpha$ is of genus $0$.
An element $b>a$ such that $\alpha(b) >  \alpha(a)$ or $\alpha(b) \leq a$
exists also since there is at least an element $j > a$ such that
$\alpha(j) = 1$. The minimality requirement stated in
conditions (1)  and (2) guarantees the uniqueness of $a$ and $b$.
Afterward, $c$ and $d$ are given by 
modulo $n$ subtractions that can be performed in exactly one way. 
It remains to show that $(a,b,c,d)$ is a sequence of separating
points. To show that $ a < b  \leq c < d$ holds, notice that if for all
$i$ such that $a < i < \alpha(a)$  we have $\alpha(i) < \alpha(a) $ 
then one of $\alpha_1$ or $\alpha_2$ given in Lemma~\ref{lem:split}
will have genus $0$ and hence contain a trivial cycle, contradicting the
fact that $\alpha$ is reduced. To show that
$g(\zeta_n(a,c)(b,d),\alpha)=0$, observe first that $a=\alpha^{-1}\zeta_n(c)$
and $c$ belong to the same cycle of $\alpha^{-1}\zeta_n$, similarly 
 $b=\alpha^{-1}\zeta_n(d)$ and $d$ belong to the same cycle of
$\alpha^{-1}\zeta_n$. Moreover, by $a=\alpha^{-1}\zeta_n(c)$, the cycle
decomposition of $\alpha^{-1}\zeta_n(a,c)$ is obtained by deleting $a$
from the cycle of $\alpha^{-1}\zeta_n$ containing it and turning it into a
fixed point. Thus $b$ and $d$ are also on the same cycle of
$\alpha^{-1}\zeta_n(a,c)$. Using Lemma~\ref{lemma:transpositions} twice
we obtain that
$z(\alpha^{-1}\zeta_n(a,c)(b,d))=z(\alpha^{-1}\zeta_n)+2$.  
\end{proof}

\begin{proposition}
\label{prop:existCanonical}
Let $\alpha=\Phi(\beta,\gamma)$ be the representation  of the reduced
permutation  $\alpha$ of genus $1$ induced by its canonical sequence of
separating points $(a,b,c,d)$. This representation has the following
properties:  

\begin{enumerate}
\item $a < b \leq c < d$ and $\alpha(a) \equiv c+1, \alpha(b) \equiv d+1
  \mod \  n$. 
\item If $x$ and $\alpha(x)$ are in  the same subset $A, B, C,$ or  $D$ then
  $\alpha(x) \equiv x+1 ({\rm mod}\  n)$.
\item There is no cycle of $\alpha$ containing elements in both $A $ and $D$ except the one
containing $b$ and $d+1$.
\item There is no cycle of $\alpha$ containing elements in both $B$ and
  $D$ except if this cycle is twisted and contains $b \in D, d+1 \in A$
  and an element $x \in B$.   
\end{enumerate}
\end{proposition}

\begin{proof}

(1) is a direct consequence of Definition~\ref{def:cancs} and
  Proposition~\ref{prop:cancs}. 

\medskip
(2) Comes from the fact that if $x$ and $\alpha(x)$ are in the same
color class $X$, and $\alpha(x) \not\equiv x+1$ then, by Lemma
\ref{lemma:planarPartition}, there is a trivial cycle of $\beta$  which
contains consecutive points in $X$, giving rise to a trivial cycle of $\alpha$
thus contradicting the fact that the permutation $\alpha$ is reduced.  

\medskip
To prove (3)  observe that if there is a cycle bicolored by $A$ and $D$
then there is an element $x$ of this cycle such that $x \in D$ and
$\alpha(x)$ in $A$. But all elements in $D$ are less than or equal to $b$, so
that  $x \neq b$ would  contradict the fact that $b$  was chosen as the
smallest such that $\alpha(b) \in A\cup B$. 

\medskip
For (4), if there is a cycle containing elements in $B$ and $D$ this
implies that there is an element $x$ in $D$ such that $\alpha(x) \in A \cup B$. As above $x = b$. And the cycle contains elements
in $A, B, D$ hence it is twisted.
\end{proof}

\begin{proposition}
Let $\alpha$ be a reduced permutation of genus $1$, represented as
$\alpha=\Phi(\beta,\gamma)$ by a four-colored noncrossing partition.
If this representation satisfies the properties stated in
Proposition~\ref{prop:existCanonical} then it is the representation
induced by the canonical sequence of separating points. 
\end{proposition}

\begin{proof}
Let $(A,B,C,D)$ denote the
sequence of color sets of the coloring induced by the canonical sequence 
$(a,b,c,d)$ of separating points via (\ref{eq:newABCD}). Suppose that
there exists another representation induced by the the 
sequence of separating points $(a', b', c', d')$ satisfying the
properties stated in in Proposition~\ref{prop:existCanonical}, and let $(A',
B',C',D')$ denote the the sequence of sets of colors in the induced 
coloring. Then $a= a'$ since both are the smallest $x$ such
that $\alpha(x) \neq x+1$, this gives also $c = c' = \alpha(a) -1$. 
Observe next that  $b \leq b'$ since $\alpha(b') \equiv d'+1 \mod n$ satisfies
$ \alpha(b')  = 1 $ or  $\alpha(b') > \alpha(a)$, thus it is not in the
interval $[a+1,\alpha(a)]$ and $b$ is the smallest integer with this
property. It suffices to show that $b$ can not be strictly less than
$b'$, afterward $d=d'$ follows from the fact that both are congruent to
$\alpha(b)-1$ modulo $n$.

Assume, by way of contradiction, that $b<b'$. 
As a consequence of $a<b<b'$, we must have $b \in D'$ 
since all the elements  in $B'$ and $C'$ are greater than $b'$ and those
in $A'\cap [1,b']$ are less than $a$ hence satisfy $\alpha(x) = x+1$. 
By property (3) in Proposition~\ref{prop:existCanonical} we can
not have $\alpha(b)\in D'$ and by property (4) we can not have
$\alpha(b)\in A'$ either. If $\alpha(b)\in D'$ then, by property (2), we
must have $d+1=\alpha(b)=b+1$, in contradiction with $b<d$. Finally, if
$\alpha(b)\in C'=[b'+1,c']\subseteq [a'+1, \alpha(a')-1]$ then
$\alpha(b)$ is not outside the interval $[a+1,\alpha(a)]$, in
contradiction with the definition of a canonical sequence of separating points. 
\end{proof}

\begin{corollary}
\label{cor:canonicalc}
The canonical four-colored noncrossing partition representation of a reduced
permutation $\alpha$ of genus $1$ may be equivalently defined by
requiring that the sequence of separating points inducing it must
satisfy the four conditions of Proposition~\ref{prop:existCanonical}. 
\end{corollary}

\section{Counting reduced partitions and permutations}
\label{sec:enured}

\subsection{Counting reduced partitions of genus $1$}

\begin{lemma}
\label{lem:bijectionMatching}
A reduced partition of genus $1$ having $k$ parts is determined by a
subset  of $2k$  integers in $\{1,2,\ldots, n\}$
and a sequence of four non-negative integers whose sum is $k-2$. 
\end{lemma}

\begin{figure}[h]
\includegraphics{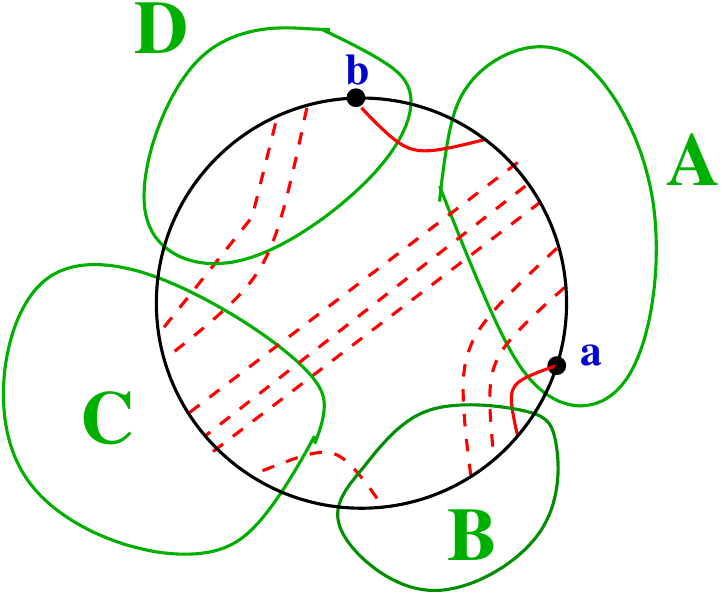}
\caption{A reduced partition}
\label{fig:reducedPart}
\end{figure}

\begin{proof}
By Corollary~\ref{c:genus1part} and as a consequence of Lemma~\ref{lemma:planarPartition}, in the canonical representation of a reduced
partition each part is bicolored and contains exactly  two points $x_i$
and $y_i$ such that $\alpha(x_i) \neq x_i+1$  and $\alpha(y_i) \neq y_i+1$. 
There is exactly one part bicolored by $A$ and $B$ that contains $a,
c+1$ and exactly one part bicolored $A, D$ that contains $b, d+1$. There
is no other part bicolored by $A,D$ and there is no part bicolored by $D,B$.
The partition is determined by the elements $x_i, y_i$ and  by the
numbers of the parts bicolored by $(A, B)$, $(A, C)$, $(B, C)$, or $(C,
D)$, respectively, see Figure~\ref{fig:reducedPart}.   

\end{proof}

\begin{theorem}
\label{thm:rnk}
The number $r_0(n,k)$ of reduced partitions of genus $1$, of the set
$\{1,\ldots,n\}$, having $k$ blocks is 
$$r_0(n,k) = \binom{n}{2k}\binom{k+1}{3}.$$
Moreover, the ordinary generating function of these partitions is given by
\begin{equation}
\label{eq:rpgf}
R_0(x,y)=\sum_{n,k \geq 0} r_0(n,k) x^ny^k\ = \ \frac{y^2x^4(1-x)^3}{\left((1-x)^2 - yx^2\right)^4}.   
\end{equation}
\end{theorem}
\begin{proof}
To obtain the first part, observe that there are $\binom{n}{2k}$ ways to
select the $2k$ integers and that the number $k-2$ may be written in 
$\binom{k+1}{3}$ ways as the sum of four non-negative integers.

\medskip

To obtain a formula for the generating function, we will use the
following variant of the binomial series formula for $(1-u)^{-m-1}$:
\begin{equation}
\label{eq:binomial}
\sum_{n=m}^{\infty} \binom{n}{m}
u^n=\frac{u^m}{(1-u)^{m+1}}\quad\mbox{holds for all $m\in {\mathbb N}$.} 
\end{equation}
Using this formula first for $u=x$ and $m=2k$ we obtain
\begin{eqnarray*}
R_0(x,y)&=&\displaystyle 
\sum_{k\geq 2}\binom{k+1}{3} y^k\sum_{n\geq 2k}\binom{n}{2k}x^n
=\sum_{k\geq 2}\binom{k+1}{3} y^k \frac{x^{2k}}{(1-x)^{2k+1}}\\
&=&
\frac{(1-x)}{yx^2}\sum_{k\geq 2} \binom{k+1}{3}
  \left(\frac{yx^2}{(1-x)^2}\right)^{k+1}. 
\\
\end{eqnarray*}
(The last part is a product of formal Laurent series.)
Substituting now $u=yx^2/(1-x)^2$ and $m=3$ into (\ref{eq:binomial})
  yields
$$
R_0(x,y)= \frac{(1-x)}{yx^2}\cdot 
\frac{\left(\frac{yx^2}{(1-x)^2}\right)^2}
{\left(1-\frac{yx^2}{(1-x)^2}\right)^4}. 
$$
Simplifying by the factors of $(1-x)$ yields the stated formula.
\end{proof}

Substituting $y=1$ in (\ref{eq:rpgf}) 
allows us to find the ordinary generating function of all reduced genus
one partitions of a given size, regardless of the number of blocks.

\begin{corollary}
\label{cor:rngf}
Let $r_0(n)$ be the number of all reduced genus $1$ partitions on
$\{1,\ldots,n\}$. Then the generating function $R_0(x) =\sum_{n\geq 4}
r_0(n) x^n$ is given by 
$$R_0(x)=x^4\frac{(1-x)^3}{(1-2x)^4}.$$
\end{corollary}
As a consequence, the ordinary generating function of the sequence
$r_0(4),r_0(5),\ldots$ is $(1-x)^3/(1-2x)^4$. This sequence
is listed as sequence A049612 in the Encyclopedia of Integer
Sequences~\cite{OEIS}. It is noted in~\cite{OEIS} that the same numbers
appear as the third row of the array given as sequence
A049600. Essentially the same array is called  
the array of {\em asymmetric Delannoy numbers $\adn_{m,n}$}
in~\cite{Hetyei-dn} where they are defined as the number of lattice paths
from $(0,0)$ to $(m,n+1)$ having steps $(x,y)\in {\mathbb N}\times
{\mathbb P}$. (Here ${\mathbb P}$ denotes the set of positive integers.)
Using \cite[Lemma 3.2]{Hetyei-dn}, it is easy to show the
following formula:
\begin{equation}
\label{eq:rn}
r_0(n)=\adn_{3,n-4}=2^{n-4}+3\binom{n-4}{1} 2^{n-5}+3\binom{n-4}{2} 2^{n-6} 
+\binom{n-4}{3} 2^{n-7}.
\end{equation}

\subsection{Counting reduced permutations of genus $1$}

\begin{theorem}
\label{thm:rnkp}
The number of reduced permutations of genus $1$ of $\Sym{n}$ with $k$ cycles is 
equal to:
\[
r_*(n,k) = \binom{n+2}{2k+2}\binom{k+1}{3} + \binom{n+1}{2k+2}\binom{k+1}{2}.
\]
More precisely, for $j=0,1,2$, the number $r_j(n,k)$ of reduced
permutations of genus $1$ of $\Sym{n}$ with $j$ back points and $k$
cycles is given by the following formulas:
$$
r_0(n,k)=\binom{n}{2k}\binom{k+1}{3},\quad
r_2(n,k)=\binom{n}{2k+2}\binom{k+2}{3}\quad \mbox{and} 
$$
$$
r_1(n,k)=\binom{n}{2k+1}\left(\binom{k+2}{3}  +   \binom{k+1}{3}
\right).  
$$
\end{theorem}
\begin{proof}
We count the four types of permutations listed in
Proposition~\ref{prop:4types},  in similar manner as we counted the
partitions of genus $1$.

\begin{enumerate}

\item The reduced permutations with no twisted cycles. These correspond
 to the partitions, their number is given in Theorem~\ref{thm:rnk}.

\item The reduced permutations with two back points. These may belong to
  the same doubly twisted cycle, or on two separate simply twisted cycles.
Let us count first the permutations with one doubly twisted cycle.

\begin{figure}[h]
\includegraphics{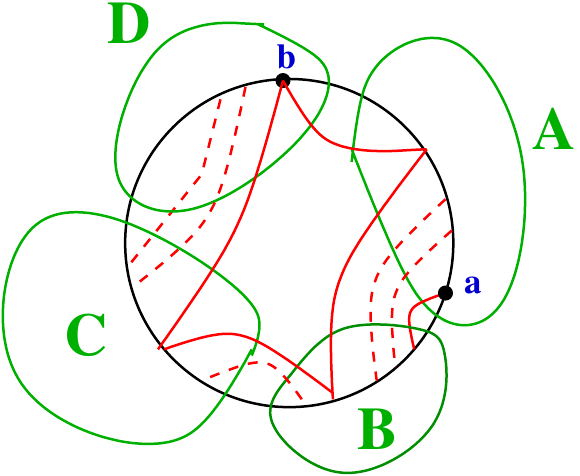}
\caption{Reduced permutation with one doubly twisted cycle}
\label{fig:reducedPerms1}
\end{figure}

The general shape of such a permutation is represented in
Figure~\ref{fig:reducedPerms1}. 
Note that the number of points $i$ such that $\alpha(i) \neq i+1$ is $4$
for the doubly 
twisted cycle and $2$ for each of the $k-1$ non-twisted cycles, giving a
total number of $2k+2$ cycles.  Moreover knowing these points the 
permutation is completely determined by the number of bicolored cycles having points in $(A,B), (B,C)
(C,D)$ so that a sequence of three non-negative integers with sum equal to $k-1$. Since the number 
of such sequences is $\binom{k+1}{2}$, the number of such permutations is:
 $$\binom{n}{2k+2}\binom{k+1}{2}.$$
Next we count the reduced permutations with two simply twisted  cycles. 
\begin{figure}[h]
\includegraphics{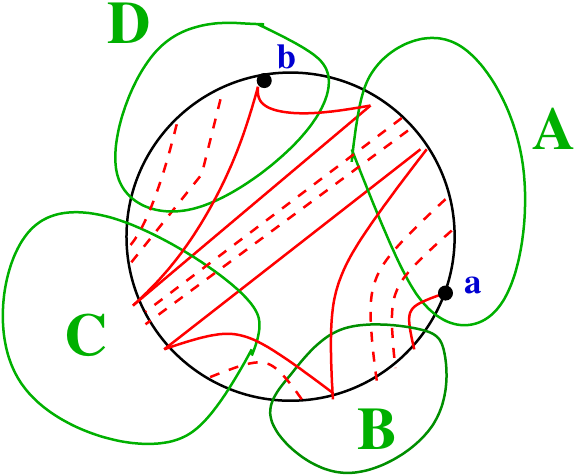}
\caption{Reduced permutation with two simply twisted cycles}
\label{fig:reducedPerms2}
\end{figure}
The general shape of such a permutation is represented in
Figure~\ref{fig:reducedPerms2}. 
Note that the number of points $i$ such that $\alpha(i) \neq i+1$ is $3$
for each of the two simply twisted cycles and $2$ for each of the $k-2$ non
twisted cycles giving a total number of $2k+2$ such points.  Moreover,
if we know these points then the  
permutation is completely determined by the number of bicolored cycles
having points in $(A,B), (B,C) (A,C), (C,D)$ so that a sequence of four
non-negative integers  with sum equal to $k-2$. Since the number 
of such sequences is $\binom{k+1}{3}$ the number of such permutations is:
 $$\binom{n}{2k+2}\binom{k+1}{3}.$$
We obtained that the number of all reduced permutations with two back
points is 
$$r_2(n,k)=\binom{n}{2k+2}\binom{k+1}{2}+\binom{n}{2k+2}\binom{k+1}{3},$$
and the stated equality follows from Pascal's formula.

\item The reduced permutations with only  one  simply twisted cycle.
\begin{figure}[h]
\includegraphics{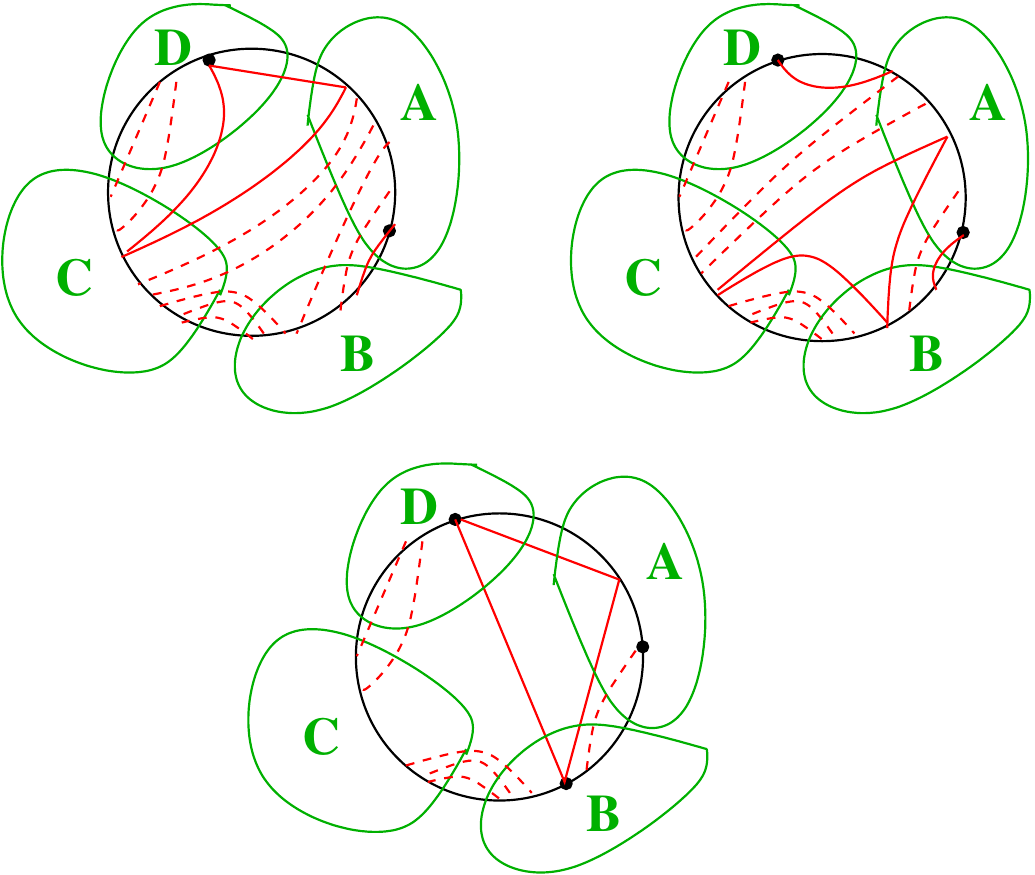}
\caption{Three reduced permutations with one simply twisted cycle}
\label{fig:reducedPerms4}
\end{figure}

The general shape of such a permutation is  represented in Figure
\ref{fig:reducedPerms4}. 
There are three different cases depending on whether the twisted cycle
is colored by $A, B, C$, or $A, C, D$ or $A, B, D$.
Note that the number of points $i$ such that $\alpha(i) \neq i+1$ is $3$
for the  simply twisted cycle and $2$ for each of the the $k-1$
non-twisted cycles, giving a total number of $2k+1$ such points.
Moreover, if we know these points 
then the permutation is completely determined by the number of bicolored
cycles having points in $(A,B), (B,C) (A,C), (C,D)$ in the two first
situations so that a sequence of four non-negative integers 
 with sum equal to $k-2$, is necessary since in the first case there is
 one cycle colored $(A, D)$ and in the second one a cycle colored
 $A,B$. In the third situation there are no cycles with elements colored
 $A,C$ so that there only 3 non-negative integers need to be known. 
 So that the number of such sequences is $\binom{k+1}{3}$ in the first
 two cases and  $\binom{k+1}{2}$ in the third one. 
in the and  the number of such permutations is:
 $$r_1(n,k)=\binom{n}{2k+1}\left(2 \binom{k+1}{3}  +   \binom{k+1}{2} \right) $$
The stated equality follows by Pascal's formula.
\end{enumerate}
Finally, adding the equations for the $r_j(n,k)$ yields 
$$
r_*(n,k)
=\binom{n}{2k}\binom{k+1}{3}
 +\binom{n}{2k+1}\left(\binom{k+2}{3}+ \binom{k+1}{3}\right)
 +\binom{n}{2k+2}\binom{k+2}{3}. 
$$
Using Pascal's formula two more times yields the stated result.
\end{proof}

\begin{proposition}
\label{prop:rgenp}
The ordinary generating function for the reduced permutations of genus $1$,
counting the number of points  and cycles, is given by:
$$
R_*(x,y)=\frac{yx^3(1-x)^2(1-x+xy)}{((1-x)^2-yx^2)^4}.
$$
More precisely, for $j=0,1,2$, the ordinary generating function for the
reduced permutations of genus $1$ with $j$ back points,
counting the number of points  and cycles, is given by:
$$
R_0(x,y)=\frac{y^2x^4(1-x)^3}{\left((1-x)^2 - yx^2\right)^4},
\quad
R_2(x,y)=\frac{yx^4(1-x)^3}{\left((1-x)^2 - yx^2\right)^4}
\quad\mbox{and}
$$
$$
R_1(x,y)=\frac{yx^3(1-x)^2((1-x)^2+yx^2)}{\left((1-x)^2 - yx^2\right)^4}.
$$
\end{proposition}
\begin{proof}
We derive our formulas from the expressions for the numbers $r_j(n,k)$
stated in Theorem~\ref{thm:rnkp}. The formula for $R_0(x,y)$ was shown in
the proof of Theorem~\ref{thm:rnk}. Comparing the expressions for
$r_0(n,k)$ and $r_2(n,k)$ yields
$r_2(n,k)=r_0(n,k+1)$, implying $yR_2(x,y)=R_0(x,y)$. We are left to
show the formula for $R_1(x,y)$, the formula for $R_*(x,y)$ may then be
obtained by taking the sum of the equations for $R_j(x,y)$ where $j=0,1,2$.

We may derive the formula for $R_1(x,y)$ in a way that is completely
analogous to the computation $R_0(x,y)$ given in the proof of
Theorem~\ref{thm:rnk}, using (\ref{eq:binomial}) several times, as
outlined below:  
\begin{eqnarray*}
R_1(x,y)&=&\displaystyle 
\sum_{k\geq 1}\left(\binom{k+2}{3}+\binom{k+1}{3}\right)  y^k
  \sum_{n\geq 2k+1}\binom{n}{2k+1} x^n\\
&=& \sum_{k\geq 1}\left(\binom{k+2}{3}+\binom{k+1}{3}\right)  y^k 
  \frac{x^{2k+1}}{(1-x)^{2k+2}}\\
&=& \frac{(1-x)^2}{y^2x^3} \sum_{k\geq 1} \binom{k+2}{3} 
        \left(\frac{yx^2}{(1-x)^2}\right)^{k+2}
   + \frac{1}{yx} \sum_{k\geq 2} \binom{k+1}{3} 
        \left(\frac{yx^2}{(1-x)^2}\right)^{k+1}\\
&=& \left(\frac{(1-x)^2}{y^2x^3}+ \frac{1}{yx}\right)\cdot 
        \frac{\left(\frac{yx^2}{(1-x)^2}\right)^3}{\left(1-\frac{yx^2}{(1-x)^2}\right)^4}.
\end{eqnarray*}
Simplifying by the factors of $(1-x)$ yields the stated formula.
\end{proof}

\section{Reducing permutations and reinserting trivial cycles}
\label{sec:red}

To count all partitions and permutations of genus $1$ we first count the
reduced objects in each class, and then count all objects obtained by
inserting trivial cycles (see Definition~\ref{def:trivcycle}) in all
possible ways. In this section we describe in general how such a
counting process may be performed. 

\begin{definition}
A {\em trivial reduction} $\pi'$ of a permutation $\pi$ of $\{1,2,\ldots,n\}$
is a permutation obtained from $\pi$ by removing a trivial cycle
$(i,i+1,\ldots,j)$ and decreasing all $k\in \{j+1,j+1,\ldots,n\}$ by 
$j-\min(0,i-1)$ in the cycle decomposition of $\pi$.  
\end{definition}
Note that a trivial cycle may contain $n$, followed by $1$, in the case
when $i>j$, and that a trivial cycle may also consist of a single fixed
point when $i=j$. Clearly $\pi'$ is a permutation of $\{1,\ldots,n'\}$ for
$n' = n-|\{i,i+1,\ldots,j\}|$ and has the same 
    genus (if we replace $\zeta_n$ with $\zeta_{n'}$). Conversely we
    will say that $\pi$ is a {\em trivial extension} (or 
    an extension) of $\pi'$. For example, a trivial reduction of
    $(1,6)(2,3,4)(5,7)$ is $(1,3)(2,4)$. Clearly a permutation is {\em reduced} 
exactly when it has no trivial reduction. In order to avoid having to treat
permutations of genus zero differently, we postulate that {\em the empty
  permutation is a reduced permutation of the empty set}.

\begin{proposition}
\label{prop:uniqred}
For any permutation $\pi$ of positive genus there is a unique reduced
permutation $\pi'$ that may be obtained by performing a sequence of
reductions on $\pi$. If $\pi$ has genus zero then this reduced permutation
is the empty permutation on the empty set. 
\end{proposition}
\begin{proof}
There is at least one reduced permutation that we may reach by
performing reductions until no reduction is possible. We only need to
prove the uniqueness of the resulting permutation.

Let us call a cycle $(i_1,\ldots, i_k)$ of $\pi$ {\em removable} if it
has the following properties:
\begin{enumerate}
\item the cyclic order of the elements $(i_1,\ldots,i_k)$ is the
  restriction of the cyclic order $\zeta$ to the set
  $\{i_1,\ldots,i_k\}$;
\item no other cycle of $\pi$ crosses $(i_1,\ldots,i_k)$;
\item the cycles whose elements belong to one of the arcs $[i_1, i_2]$, $[i_2,
  i_3]$, \ldots, or $[i_{k-1}, i_k]$ are not twisted;
\item no cycle whose elements belong to one of the arcs $[i_1, i_2]$, $[i_2,
  i_3]$, \ldots, or $[i_{k-1}, i_k]$ crosses any other cycle of $\pi$.
\end{enumerate}
We claim that a cycle of $\pi$ gets removed in any and every reduction
process that leads to a reduced permutation, exactly when $\pi$ is
removable. On the one hand it is easy to see directly that any cycle
that gets removed in the reduction process must be removable: assume 
after a certain number of reductions, the cycle $(i_1,\ldots,i_k)$
becomes the trivial cycle $(i,i+1,\ldots,j)$ where $i_1$ corresponds to
$i_1$. Applying a reduction or an extension does not change the fact
whether a cycle, present in both permutation is obtained by the
restricting the cyclic order of all elements, this proves property (1).
Neither the previously removed cycles, nor the cycles surviving after
the removal of $(i_1,\ldots,i_k)$ can cross $(i_1,\ldots,i_k)$. The last
two properties follow from the fact that the cycles whose elements 
belong to one of the arcs $[i_1, i_2]$, $[i_2,  i_3]$, \ldots, or
$[i_{k-1}, i_k]$ all become trivial cycles in the reduction process.

On the other hand, it is easy to show by induction on the
number of cycles located on the arcs $[i_1, i_2]$, $[i_2,
  i_3]$, \ldots, $[i_{k-1}, i_k]$ of a removable cycle that every
removable cycle ends up being removed in the reduction process. The
basis of this induction is that a removable cycle containing no other
cycles on its arcs is trivial. Any other removable cycle becomes trivial
after the removal of all cycles contained on the arcs $[i_1, i_2]$, $[i_2,
  i_3]$, \ldots, $[i_{k-1}, i_k]$: these cycles are easily seen to be
removable due to properties (3) and (4) and, if we list the elements of
each such cycle $(j_1,\ldots,j_l)$ in the order they appear on the
respective arc $[i_s,i_{s+1}]$, then the set of cycles contained 
on the arcs $[j_1,j_2]$, \ldots, $[j_{l-1},j_l]$ is a proper subset of the
cycles contained on the arc $[i_s,i_{s+1}]$. The induction hypothesis
thus becomes applicable.

We found that the exact same cycles get removed in every reduction
process that yields a reduced permutation, even if the order of the
reduction steps may vary. After each reduction step, the surviving
elements get relabeled, and the new label depends on the actual
reduction step. However, it is easy to find the final label of each
element $i$ located in a cycle that ``survives'' the entire reduction
process: $i$ gets decreased exactly by the number of all elements of
$\{1,\ldots,i-1\}$ that belong to a removable cycle.

Clearly a permutation has genus zero exactly when all of its cycles are
removable. 
\end{proof}
As a consequence of Proposition~\ref{prop:uniqred}, if a class of
permutations is closed under reductions and extensions then we are able
to describe this class reasonably well by describing the reduced
permutations in the class. Examples of such permutation classes include:
\begin{itemize}
\item[--] the class of all partitions;
\item[--] the class of all permutations of a given genus;
\item[--] the class of all partitions of a given genus.
\end{itemize}   
The main result of this section shows that knowing the reduced
permutations allows not only to {\em describe} but also to {\em count}
the permutation in the class closed under reductions and extensions that
they generate.
To state our main
result we will need to use the generating function 
\begin{equation}
\label{eq:dxy}
D(x,y)=\frac{1-x-xy-\sqrt{(x+xy-1)^2-4x^2y}}{2\cdot x}+1
\end{equation}
of noncrossing partitions. This function is the formal power series
solution of the quadratic equation
\begin{equation}
\label{eq:Deq}
D(x,y)=1+ xy\cdot D(x,y)+ x\cdot (D(x,y)-1)D(x,y),
\end{equation}
whose other solution is only a formal Laurent series. As it is
well-known \cite[sequence A001263]{OEIS}, $[x^ny^k] 
D(x,y)$ is the  number of noncrossing partitions of the set
$\{1,\ldots,n\}$ having $k$ parts. Note that we deviate from the usual
conventions by defining the constant term
to be $1$, i.e. we consider that there is one noncrossing partition on
the empty set and it has zero blocks. Our main result is the following.
\begin{theorem}
\label{thm:pd}
Consider a class ${\mathcal C}$ of permutations that is closed under 
trivial reductions and extensions. Let $p(n,k)$ and $r(n,k)$
respectively be the number of all, respectively all reduced permutations
of $\{1,\ldots,n\}$ in the class having $k$ cycles. Then the generating
functions $P(x,y) =\sum_{n,k} p(n,k) x^ny^k$ and
$R(x,y) =\sum_{n,k} r(n,k) x^ny^k$ satisfy the equation
$$
P(x,y)=R(x\cdot D(x,y),y)\cdot 
\left(1+x\cdot \frac{\frac{\partial}{\partial x} D(x,y)}{D(x,y)}\right). 
$$
Here $D(x,y)$ is the generating function of noncrossing partitions given
in \eqref{eq:dxy}.
\end{theorem}
\begin{proof} Consider an arbitrary permutation $\pi$ of
  $\{1,\ldots,n\}$ in the class having $k$ cycles. We distinguish two cases, and
  describe the generating function of the permutations belonging to each
  case. The term ``removable cycle'' we use here is the one that was
  defined in the proof of Proposition~\ref{prop:uniqred}. 

\bigskip
\noindent{\bf Case 1} {\em The element $1$ does not belong to a
  removable cycle.} After reducing 
the permutation to the reduced permutation $\pi'$, we obtain a reduced
permutation on the set 
$\{1,\ldots,n_1\}$ having $k_1$ blocks for some $n_1\leq n$ and $k_1\leq
k$. The cycles of $\pi$ permutation that were removed have $n-n_1$
elements, and they form $n_1$ noncrossing partitions  
on the arcs created by the elements appearing in $\pi'$. They also
have $k-k_1$ blocks.  Thus there
are exactly $[x^{n-n_1}y^{k-k_1}] D(x,y)^{n_1}$ permutations
that may be reduced to the same reduced partition. 
The number of permutations counted in this case is
$$
\sum_{n_1\geq 4}\sum_{k_1\geq 2}  r(n_1,k_1) [x^{n-n_1}y^{k-k_1}] D(x,y)^{n_1}
$$
Using the fact that, for any formal power series $f(x,y)$, 
$[x^{n-n1}y^{k-k_1}]f(x,y)$ is the same as $[x^ny^k] x^{n_1}y^{k_1}
f(x),y$, we see that the above sum is exactly the coefficient of
$x^ny^k$ in $R(x\cdot D(x,y),y)$. 

\bigskip
\noindent{\bf Case 2} {\em The element $1$ belongs to a removable
  cycle.} Let $j+1$, respectively $i-1$ be the smallest, respectively
largest element that does not belong to a removable cycle. The arc 
$\{i,i+1,\ldots,n,1,\ldots,j\}$ is then a union of elements of removable
cycles. (Here we allow
$i-1=n$, then $i=1$ and $n$ does not belong to the arc). 
Let us denote the number of elements
of this arc by $n_2$ and assume that the noncrossing partition formed by
the removable cycles whose elements belong to this
arc has $k_2$ blocks. As in the previous case, let $n_1$ be the number of
elements belonging to not removable cycles, and assume that there are $k_1$
not removable cycles. There are $r(n_1,k_1)$ ways to select the
reduced permutation, $[x^{n_2}y^{k_2}] D(x,y)$ ways to select the
noncrossing partition on the arc $\{i,i+1,\ldots,n,1,\ldots,j\}$
containing $1$, and $n_2$ ways to select the position of $1$ in its
arc. We need to fill in the remaining $n-n_1-n_2$ elements of removable cycles
and group them into noncrossing partitions on the $n_1-1$ other
arcs created by the $n_1$ elements of not removable cycles. We also need to
make sure that the number of these other removable cycles is $k-k_1-k_2$. 
The number of permutations counted in this case is
$$
\sum_{n_1\geq 4}\sum_{k_1\geq 2}\sum_{n_2\geq 1}\sum_{k_2\geq 1}  
    r(n_1,k_1) \left(n_2 [x^{n_2}y^{k_2}] D(x,y)\right)\cdot
    \left([x^{n-n_1-n_2}y^{k-k_1-k_2}] D(x,y)^{n_1-1}\right).  
$$
Note that $n_2 [x^{n_2}y^{k_2}] D(x,y)$ in the above sum is the
coefficient of $x^{n_2}y^{k_2}$ in $x\cdot \frac{\partial}{\partial x}
D(x,y)$. Using the same observation as at the end of the previous case,
we obtain that the number of partitions counted in this case is 
$$
[x^ny^k]\left( R(x\cdot D(x,y),y)\cdot x\cdot \frac{\frac{\partial}{\partial
    x} D(x,y)}{D(x,y)}\right). 
$$
\end{proof}
We conclude this section with rewriting the factor 
$1+x\cdot \frac{\partial}{\partial x} D(x,y)/D(x,y)$, appearing
in Theorem~\ref{thm:pd}, in an equivalent form.
\begin{proposition}
\label{prop:dxd}
$$
1+x\cdot \frac{\frac{\partial}{\partial x} D(x,y)}{D(x,y)}
=
\frac{1-xD(x,y)}{\sqrt{(x+xy-1)^2-4x^2y}}.
$$
\end{proposition}
\begin{proof}
We may rewrite \eqref{eq:Deq} as 
$$
x\cdot D(x,y)^2+(xy-1-x) D(x,y)+1=0.
$$
Taking the partial derivative with respect to $x$ on both sides we obtain 
$$
D(x,y)^2-2xD(x,y)\frac{\partial}{\partial x}
D(x,y)-(1-y)D(x,y)-(1+x-xy)\frac{\partial}{\partial x} D(x,y)=0.
$$
Using this equation we may express $\frac{\partial}{\partial x} D(x,y)$
as follows:
\begin{equation}
\label{eq:Dpart1}
\frac{\partial}{\partial x} D(x,y)=\frac{D(x,y)(D(x,y)+y-1)}{1+x-xy-2xD(x,y)}.
\end{equation}
This equation directly implies
\begin{equation}
\label{eq:dfactor}
1+\frac{x \frac{\partial}{\partial x} D(x,y)}{D(x,y)}=
\frac{1-xD(x,y)}{1+x-xy-2xD(x,y)}.
\end{equation}
Finally, as a direct consequence of \eqref{eq:dxy} we have 
\begin{equation}
\label{eq:discr}
1+x-xy-2xD(x,y)=\sqrt{(x+xy-1)^2-4x^2y}.
\end{equation}
Combining \eqref{eq:dfactor} and \eqref{eq:discr} yields the stated equality.
\end{proof}
\begin{corollary}
\label{cor:pd}
The formula stated in Theorem~\ref{thm:pd} is equivalent to stating
$$
P(x,y)=R(x\cdot D(x,y),y)\cdot 
\frac{1-xD(x,y)}{\sqrt{(x+xy-1)^2-4x^2y}}.
$$
\end{corollary}
\begin{remark}
The numbers 
$$J(n,k) =[x^ny^k] \left(x\cdot \frac{\frac{\partial}{\partial x}
  D(x,y)}{D(x,y)}\right)$$  
are tabulated as entry A103371 in~\cite{OEIS}.
It is stated in the work of A.\ Laradji and A.\ Umar~\cite[Corollary
  3.10]{Laradji-Umar} referenced therein, that 
$$
J(n,k)=\binom{n}{k}\binom{n-1}{k-1}.
$$
\end{remark}

\section{Counting all partitions and permutations of genus one}
\label{sec:all}

In this section we find the ordinary generating function for the numbers
$p_0(n,k)$  of all partitions of genus one the set $\{1,\ldots,n\}$,
having $k$ parts, and prove an analogous result for permutations of
genus $1$. Our main result is the following.
\begin{theorem}
\label{thm:pnkgf}
Let the number $p_0(n,k)$ of all partitions of $\{1,\ldots,n\}$ of genus
one having $k$ parts. Then the generating function
$$
P_0(x,y) =\sum_{n\geq 4}\sum_{k\geq 2} p_0(n,k) x^n y^k
$$
is given by the equation
$$
P_0(x,y)=\frac{x^4y^2}{(1-2(1+y)x+x^2(1-y)^2)^{5/2}}.
$$
\end{theorem}
We will see in Section~\ref{sec:yip} that Theorem~\ref{thm:pnkgf} is
equivalent to an explicit formula~\eqref{eq:pnk} for the numbers
$p_0(n,k)$, originally conjectured by M.\ Yip~\cite[Conjecture 3.15]{Yip}. 
We will prove Theorem~\ref{thm:pnkgf} by combining 
Theorem~\ref{thm:pd} with the formula (\ref{eq:rpgf}) for the generating
function  $R_0(x,y)$ of reduced partitions of genus one. We use the
equivalent form of Theorem~\ref{thm:pd} stated in Corollary~\ref{cor:pd}
and use Propositions~\ref{prop:RxDy} below to simplify 
$R_0(x\cdot D(x,y),y)$. Theorem~\ref{thm:pnkgf} thus follows from
Theorem~\ref{thm:pd}, by multiplying the formulas given in 
Propositions~\ref{prop:dxd} and \ref{prop:RxDy}.

\begin{proposition}
\label{prop:RxDy}
The generating function $R_0(x,y)$ of reduced partitions of genus one
satisfies the equality 
$$
R_0(x\cdot D(x,y),y)
=\frac{x^4y^2}{(1-xD(x,y))((x+xy-1)^2-4x^2y)^2}.
$$
\end{proposition}
\begin{proof}
We will use $D$ as a shorthand for $D(x,y)$. 
Using (\ref{eq:rpgf}) we may write 
\begin{equation}
\label{eq:rpgfsubs}
R_0(x\cdot D,y)
=\frac{y^2x^4D^4(1-xD)^3}{((1-xD)^2-yx^2D^2)^4}
\end{equation}
An equivalent form of \eqref{eq:Deq} is 
\begin{equation}
\label{eq:xyD}
xy D=(D-1)(1-xD),
\end{equation}
which may be used to eliminate the variable $y$ in the denominator on
the right hand side of (\ref{eq:rpgfsubs}). Thus we obtain 
$$
R_0(x\cdot D,y)
=\frac{y^2x^4D^4(1-xD)^3}{((1-xD)^2-(D-1)(1-xD)xD)^4}
=\frac{y^2x^4D^4(1-xD)^3}{((1-xD)(1-xD^2))^4}.
$$
Simplifying by the factors of $(1-xD)$ yields 
$$
R_0(x\cdot D,y)
=\frac{y^2x^4D}{(1-xD)}\cdot 
\left(\frac{D}{1-xD^2}\right)^4.
$$
We are left to show that the second factor is $((x+xy-1)^2-4x^2y)^{-2}$.
By \eqref{eq:discr}, this is equivalent to showing 
$$
1+x-xy-2xD=\frac{1-xD^2}{D}
$$
which is a rearranged version of \eqref{eq:Deq}.
\end{proof}

Substituting $y=1$ into the formula given in Theorem~\ref{thm:pnkgf} has
the following consequence.
\begin{corollary}
The number of $p_0(n)$ all partitions of $\{1,\ldots,n\}$ of genus one has the
ordinary generating function 
$$
\sum_{n=4}^{\infty}
p_0(n)x^n=\frac{x^4}{(1-4x)^{5/2}}. 
$$
\end{corollary}
The coefficient of $x^n$ in the above formula is easily extracted:
\begin{corollary}
The number of all genus one partitions on $\{1,\ldots,n\}$ is 
$$
p_0(n)=\binom{-5/2}{n-4}(-1)^{n-4} 4^{n-4}=\frac{(2n-5)!}{6\cdot (n-4)!(n-3)!}.  
$$
\end{corollary}
The sequence $p_0(4),p_0(5),\ldots$ is listed as sequence A002802 in
\cite{OEIS} and referred to (essentially) as the number of permutations
of genus one.  See also \cite[formula (13)]{Walsh-Lehman}. Now we see
that partitions of genus one are counted by the same sequence, shifted by one. 

Next we follow an analogous procedure to count all permutations of genus
$1$. 
\begin{theorem}
\label{thm:pnkgfp}
Let $p_*(n,k)$ be the number  of all permutations in $\Sym{n}$ of genus
one having $k$ cycles. Then the generating function
$P_*(x,y) =\sum_{n,k} p_*(n,k) x^n y^k$
is given by the equation
$$
P_*(x,y)=\frac{x^3y}{(1-2(1+y)x+x^2(1-y)^2)^{5/2}}.
$$
More precisely, for $j=0,1,2$, let $p_j(n,k)$  be the number  of all
permutations in $\Sym{n}$ of genus one having $k$ cycles and $j$ back points.
Then the generating functions $P_j(x,y) =\sum_{n,k} p_j(n,k) x^n y^k$
are given by the formulas
$$
P_0(x,y)=\frac{x^4y^2}{(1-2(1+y)x+x^2(1-y)^2)^{5/2}}, 
$$
$$
P_2(x,y)=\frac{x^4y}{(1-2(1+y)x+x^2(1-y)^2)^{5/2}} \quad\mbox{and}
$$
$$
P_1(x,y)=\frac{x^3y(1-xy-x)}{(1-2(1+y)x+x^2(1-y)^2)^{5/2}}.
$$ 
\end{theorem}
The formula for $P_0(x,y)$ was shown in 
Theorem~\ref{thm:pnkgf} above. As noted in the proof of 
Proposition~\ref{prop:rgenp}, the generating function $R_2(x,y)$ differs
from $R_0(x,y)$ only by a factor of $y$. After reproducing the same
calculation to obtain $P_2(x,y)$ from $R_2(x,y)$, we find that 
$P_0(x,y)=yP_2(x,y)$. Therefore, to prove
Theorem~\ref{thm:pnkgfp} above, it suffices to show the formula for
$P_1(x,y)$, the equation for $P_*(x,y)$ will then arise as the sum of the
equations for the $P_j(x,y)$.

Similarly to the proof of Theorem~\ref{thm:pnkgf}, we may show this
formula by combining Corollary~\ref{cor:pd} with the formula for
$R_1(x,y)$ given in Proposition~\ref{prop:rgenp}. We may use
Propositions~\ref{prop:RxDyp} below to simplify $R_1(x\cdot D(x,y),y)$. 

\begin{proposition}
\label{prop:RxDyp}
The generating function $R_1(x,y)$ of reduced permutations of genus $1$
having one back point satisfies the equality 
$$
R_1(x\cdot D(x,y),y)
=\frac{x^3y(1-xy-x)}{(1-xD(x,y))((x+xy-1)^2-4x^2y)^2}.
$$
\end{proposition}
\begin{proof}
We will use $D$ as a shorthand for $D(x,y)$. 
Using Proposition~\ref{prop:rgenp} we may write 
$$
R_1(x\cdot D,y)
=
\frac{yx^3 D^3 (1-xD)^2((1-xD)^2+yx^2D^2)}{\left((1-xD)^2 - yx^2D^2\right)^4}
$$
Just like in the proof of Proposition~\ref{prop:RxDy} we may use (\ref{eq:xyD})
to eliminate the variable $y$ in the denominator and get 
$$
R_1(x\cdot D,y)
=
\frac{yx^3 D^3 (1-xD)^2((1-xD)^2+yx^2D^2)}{((1-xD)(1-xD^2))^4}
=
\frac{yx^3 D^3 ((1-xD)^2+yx^2D^2)}{(1-xD)^2(1-xD^2)^4}.
$$ 
We use (\ref{eq:xyD}) again to rewrite the factor $((1-xD)^2+yx^2D^2)$
in the numerator and get 
$$
R_1(x\cdot D,y)
=
\frac{yx^3 D^3 (1-2xD+x D^2)}{(1-xD)(1-xD^2)^4}=
\frac{yx^3 (1-2xD+x D^2)}{(1-xD)D} \cdot
\left(\frac{D}{(1-xD^2)}\right)^4  
$$ 
We have seen at the end of the proof of Proposition~\ref{prop:RxDy} that
the last factor is $(x+xy-1)^2-4x^2y)^{-2}$. Taking this fact into
account, comparing the last equation with the proposed statement, we
only need to show the following equality:
$$
\frac{1-2xD+x D^2}{D}=1-xy-x.
$$
This last equation is a rearranged version of \eqref{eq:Deq}.
\end{proof}

\section{Extracting the coefficients from our generating functions}
\label{sec:yip}

In this section we will show how to extract the coefficients from our
generating functions to obtain explicit formulas for the numbers of
genus $1$ partitions and permutations. Our main tool is a generalization
of the following equation.  
\begin{equation}
\label{eq:oddpower}
\frac{x^4y^2}{(1-2(1+y)x+x^2(1-y)^2)^{5/2}}=
\sum_{n\geq 4} \frac{1}{6}\binom{n}{2} x^n \sum_{k=2}^{n-2}
\binom{n-2}{k}\binom{n-2}{k-2} y^k.
\end{equation}
According to this equation, M.\ Yip's conjecture~\cite[Conjecture
  3.15]{Yip}, stating  
\begin{equation}
\label{eq:pnk}
p_0(n,k)=\frac{1}{6}\binom{n}{2}\binom{n-2}{k}\binom{n-2}{k-2}. 
\end{equation}
is equivalent to our Theorem~\ref{thm:pnkgf} and thus true.
Since, by Theorem~\ref{thm:pnkgfp},  the generating function of
genus one permutations only differs by a factor of $xy$, we also obtain
a new way to count these objects, thus providing a new proof of the
result first stated by A.\ Goupil and G.\ Schaeffer~\cite{Goupil-Schaeffer}. 

After dividing both sides by $x^4y^2$ and shifting $n$ and $k$ down by
two, we obtain the following equivalent form of
equation~\eqref{eq:oddpower}.
\begin{equation}
\label{eq:oddpower2}
\frac{1}{(1-2(1+y)x+x^2(1-y)^2)^{5/2}}=
\sum_{n\geq 2} \frac{1}{6}\binom{n+2}{2} x^{n-2}\sum_{k=0}^{n-2}
\binom{n}{k+2}\binom{n}{k} y^k.
\end{equation}
This equation is the special case (when $m=2$) of
Equation~(\ref{eq:pnkgen}) below, that 
holds for all $m\in{\mathbb N}$.
\begin{equation} 
\label{eq:pnkgen}
\frac{1}{(1-2(1+y)x+x^2(1-y)^2)^{(2m+1)/2}}=\sum_{n\geq m}\sum_{k\geq 0}
\frac{\binom{n+m}{m}\binom{n}{k}\binom{n}{m+k}}{\binom{2m}{m}}x^{n-m} y^k.
\end{equation}
Equation~(\ref{eq:pnkgen}) may be obtained from \cite[Equation
  (2)]{Gessel}, after substituting $\alpha=(2m+1)/2$ and replacing each
appearance of $y$ with $xy$ in that formula (on the right hand side, one
also needs to replace the summation indices $i$, and $j$ respectively,
with $n-m-k$ and $k$, respectively). As pointed out
by Strehl~\cite[p.\ 180]{Strehl} (see also \cite[p.\ 64]{Gessel}),
\cite[Equation (2)]{Gessel} is a consequence of classical results in the
theory of special functions.    
\begin{remark}
\label{rem:pnkgen}
Equation~(\ref{eq:pnkgen}) may also be derived directly from classical
results as follows. Take the $m$th derivative with respect to $u$ of the 
generating function $\sum_{n\geq 0} L_n(u) t^n$ of the Legendre
polynomials (given in~\cite[Ch.\ V, (2.34)]{Chihara}), multiply
both sides by $2^m/(t^mm!)$, use ~\cite[(4.21.2)]{Szego} to express
$L_n(u)$, substitute $u=(1+y)/(1-y)$ and $t=x(1-y)$, and use the
Chu-Vandermonde identity. 
\end{remark}
We conclude this section with providing explicit formulas for the
number of all permutations of genus $1$, with a given numbers of
points, cycles, and back points.

\begin{theorem}
\label{thm:allperm}
The number of all permutations of genus $1$ of $\Sym{n}$ with $k$ cycles is 
equal to:
\[
p_*(n,k) = \frac{1}{6}\binom{n+1}{2}\binom{n-1}{k+1}\binom{n-1}{k-1}
\]
More precisely, for $j=0,1,2$, the number $p_j(n,k)$ of 
permutations of genus $1$ of $\Sym{n}$ with $j$ back points and $k$
cycles is given by the following formulas:
$$
p_0(n,k)=\frac{1}{6}\binom{n}{2}\binom{n-2}{k}\binom{n-2}{k-2},\quad
p_2(n,k)=\frac{1}{6}\binom{n}{2}\binom{n-2}{k+1}\binom{n-2}{k-1}\quad
\mbox{and}  
$$
$$
p_1(n,k)=\frac{1}{3}\binom{n}{2}\binom{n-2}{k}\binom{n-2}{k-1}.  
$$
\end{theorem}
\begin{proof}
The formulas for $p_0(n,k)$, $p_2(n,k)$ and $p_*(n,k)$ are all
direct consequences of Theorems~\ref{thm:pnkgfp} and Equation~(\ref{eq:pnkgen}).
Using the same results to find $p_1(n,k)$ amounts to using the obvious equality
$$
p_1(n,k)=p_*(n,k)-(p_0(n,k)+p_2(n,k)),
$$ 
which is equivalent to showing that the sum of the stated values of
the $p_j(n,k)$ gives the stated value of $p_*(n,k)$.  For that purpose
note that 
$$
p_0(n,k)+\frac{p_1(n,k)}{2}
=\frac{1}{6}\binom{n}{2}\binom{n-2}{k}\left(\binom{n-2}{k-2}+\binom{n-2}{k-1}\right),$$
which, by Pascal's formula, gives
\begin{equation}
\label{eq:p01/2}
p_0(n,k)+\frac{p_1(n,k)}{2}
=\frac{1}{6}\binom{n}{2}\binom{n-2}{k}\binom{n-1}{k-1} 
=\frac{n}{12}(k+1)\binom{n-1}{k+1}\binom{n-1}{k-1}. 
\end{equation}
A similar use of Pascal's formula yields
\begin{equation}
\label{eq:p21/2}
p_2(n,k)+\frac{p_1(n,k)}{2}
=\frac{1}{6}\binom{n}{2}\binom{n-2}{k-1}\binom{n-1}{k+1} 
=\frac{n}{12}(n-k)\binom{n-1}{k-1}\binom{n-1}{k+1}. 
\end{equation}
The sum of (\ref{eq:p01/2}) and (\ref{eq:p21/2}) is 
$$
\sum_{j=0}^2 p_j(n,k)=\frac{n(n+1)}{12}\binom{n-1}{k-1}\binom{n-1}{k+1}, 
$$
as required.
\end{proof}

\section{Concluding remarks}
\label{sec:concl}

Our four-colored noncrossing partition representation of permutations of
genus $1$ is reminiscent of the use of three types of crossing hyperedges in the
hypermonopole diagram representing a genus $1$ partition in M.\ Yip's
Master's thesis~\cite{Yip}. This analogy becomes even more explicit at
the light of Remark~\ref{rem:3color} stating that, for partitions of
genus $1$, three colors suffice. Whereas the hypermonopole diagrams are
of topological nature (parts are represented with ``curvy lines'') our
representation is combinatorial (parts may be represented with
polygons). By better understanding the relation between the two models,
perhaps it is possible to show that every genus one partition has a
hypermonopole diagram on a torus in such a way that boundaries of
hyperedges are finite unions of  ``straight'' (circular) arcs. In either
case, non-uniqueness of the representation makes direct counting difficult.

Lemma~\ref{lemma:nbBackPoints} establishes a relationship between
$\alpha$ and $\alpha^{-1}\zeta_n$. It is worth noting that, in the case
when $g(\alpha)=0$, the permutation $\alpha^{-1}\zeta_n$ is the
permutation representing the {\em Kreweras dual} of the
noncrossing partition represented by $\alpha$. G.\ Kreweras~\cite{Kreweras}
used this correspondence to show that the lattice of noncrossing
partitions is self-dual. M.\ Yip has shown that the poset of genus
$1$ partitions is rank-symmetric~\cite[Proposition 4.5]{Yip}, but not
self dual~\cite[Proposition 4.6]{Yip} for $n\geq
6$. Lemma~\ref{lemma:nbBackPoints} suggests that maybe true duality
could be found between genus $1$ partitions and permutations with $2$
back points, after defining the proper partial order on the set of all
genus $1$ permutations. In this setting, permutations with exactly one
back point would form a self-dual subset. Their number $p_1(n,k)$, given in
Theorem~\ref{thm:allperm}, may be rewritten as 
$$
p_1(n,k)=\binom{n}{3} N(n-2,k-1),
$$ 
where $N(n-2,k-1)$ is a Narayana number. It is a tantalizing thought
that this simple formula could have a very simple proof. If this is the
case, then the formulas for $p_0(n,k)$ and $p_1(n,k)$ could be easily
derived, using Lemma~\ref{lemma:nbBackPoints} and Yip's rank-symmetry
result~\cite[Proposition 4.5]{Yip} to establish $p_2(n,k)=p_0(n,k+1)$,
and then the formula for $p_*(n,k)$ already stated by A.\ Goupil and
Shaeffer~\cite{Goupil-Schaeffer} to complete a setting in which the
formula for $p_0(n,k)$ may be shown by induction on $k$.  
A ``numerically equivalent'' conjecture (albeit for sets of partitions)
was stated by M.\ Yip~\cite[Conjecture 4.10]{Yip}.

Equation~(\ref{eq:pnkgen}) naturally inspires the question: what other
combinatorial objects are counted by the coefficients of $x^ny^k$ in the
Taylor series of 
$$(1-2(1+y)x+x^2(1-y)^2)^{-(2m+1)/2},$$ 
when $m$ is some other nonnegative integer. For $m=0$, we obtain 
$$
\frac{1}{(1-2(1+y)x+x^2(1-y)^2)^{1)/2}}=\sum_{n\geq m}\sum_{k\geq 0}
\binom{n}{k}^2 x^{n} y^k.
$$
These coefficients are listed as sequence A008459 in~\cite{OEIS}. Among
others, they count the {\em type $B$ noncrossing partitions} of rank $k$ of an
$n$-element set. In~\cite{Simion}, R.\ Simion constructed a simplicial
polytope in each dimension whose $h$ vector entries are the
squares of the binomial coefficients. The number of $j$-element faces of 
the $n$-dimensional polytope is $f_{j-1}=\binom{n+j}{j}$. Another class
of simplicial polytopes with the same face numbers was defined
in~\cite{Hetyei-L} as the class of all simplicial polytopes arising by
taking any pulling triangulation of the boundary complex of the
Legendrotope. The Legendrotope is combinatorially equivalent to the
intersection of a standard crosspolytope with any hyperplane passing through its
center that does not contain any of its vertices.
For all these polytopes the polynomial
$$
F(u)=\sum_{j=0}^n f_{j-1} \left(\frac{u-1}{2}\right)^j
$$
is a Legendre polynomial, and the squares of the
binomial coefficients are their $h$-vector entries. For higher values of
$m$, taking the $m$th derivative of $F(u)$ (see Remark~\ref{rem:pnkgen})
corresponds to summing over the links of all 
$(m-1)$ dimensional faces. It is not evident from this interpretation
why we should get integer entries, even after dividing by
$\binom{2m}{m}$, and it seems an interesting question to
see whether for the type $B$ associahedron or for some
very regular triangulation of the Legendrotope, symmetry reasons would
explain the integrality. For $m=1$, I.\ Gessel has shown~\cite{Gessel}  
that the coefficients count convex polyominoes. Finally, for general
$m$, the coefficients have a combinatorial interpretation in the work of
V.\ Strehl~\cite{Strehl-Jacobi} on Jacobi configurations. 
Even though V.\ Strehl uses exponential generating functions, the use of
the same coefficients becomes apparent by comparing his summation
formula on page 303 with \cite[Equation (2)]{Gessel}. It seems worth
exploring whether deeper connections exist between the above listed
models. 

\section*{Acknowledgments} 
The second author wishes to express his heartfelt thanks to Labri, Universit\'e
Bordeaux I, for hosting him as a visiting researcher in the of Summer
2011 and 2013. This work was partially supported by a grant from the
Simons Foundation (\#245153 to G\'abor Hetyei) and by the ANR project
(BLAN-0204.07 Magnum to Robert Cori). We wish to thank Ira Gessel for
pointing us to the right sources to shorten our calculations.

\end{document}